\documentclass[12pt]{article}

\usepackage{amsfonts}
\usepackage{amssymb}
\usepackage{amsmath}
\usepackage{enumitem}
\usepackage{float}
\usepackage{xspace}
\usepackage{geometry}
\usepackage{centernot}
\usepackage{cancel}

\usepackage[pdftex]{graphicx}

\usepackage{graphicx}

\usepackage{amsthm}
\usepackage{thmtools}

\newtheorem{lemma}{Lemma}[section]
\newtheorem{prop}[lemma]{Proposition}
\newtheorem{corr}[lemma]{Corollary}
\newtheorem{theorem}[lemma]{Theorem}
\newtheorem{result}[lemma]{Result}
\theoremstyle{definition}
\newtheorem{deff}[lemma]{Definition}

\theoremstyle{definition}

\theoremstyle{definition}

\newtheoremstyle{mythmstyle}%
    {}%
    {}%
    {\it}%
    {}%
    {\bf}%
    {}%
    { }%
    {\thmname{#1}\thmnumber{ #2}%
    \thmnote{: #3\addcontentsline{toc}{subsubsection}{{\it#1}: #3}}. }
\theoremstyle{mythmstyle}

\usepackage[capitalize]{cleveref}

\Crefname{thm}{Theorem}{Theorems}
\Crefname{lem}{Lemma}{Lemmas}
\Crefname{deff}{Definition}{Definitions}
\Crefname{prop}{Proposition}{Propositions}
\Crefname{corr}{Corollary}{Corollaries}

\usepackage{tikz}
\usetikzlibrary{arrows}
\usetikzlibrary{graphs,graphs.standard,quotes}

\usepackage{amsopn}
\DeclareMathOperator{\tra/}{trace}
\DeclareMathOperator{\ker/}{ker}
\DeclareMathOperator{\ink/}{IK}

\def\In/{$\mathcal{I}_n$}
\def\lcs/{$\mathfrak{LC}(S)$}
\def\rcs/{$\mathfrak{RC}(S)$}

\usepackage{titlesec}
\titleformat{\section}
  {\normalfont\large\scshape}{\thesection}{1em}{}

\usepackage{authblk}
\author{Matthew Brookes \thanks{Email: {\tt mdkb500@york.ac.uk}}}
\affil{{\small \em Department of Mathematics, University of York,} \\ \small{\em York, YO10 5DD}}

\usepackage{titling}

\predate{}
\postdate{}
\date{}

\pretitle{\begin{center}\huge\scshape}
\posttitle{\end{center}}

\begin{document}
\title{The Lattice of One-Sided Congruences on an Inverse Semigroup}
\maketitle

\begin{abstract}

We build on the description of left congruences on an inverse semigroup in terms of the kernel and trace due to Petrich and Rankin.  The notion of an inverse kernel for a left congruence is developed.  Various properties of the trace and inverse kernel are discussed, in particular that both the trace and inverse kernel maps are onto $\cap$-homomorphisms.  The lattice of left congruences is identified as a subset of the direct product of the lattice of congruences on the idempotents and the lattice of full inverse subsemigroups. We use this to describe the lattice of left congruences on the bicyclic monoid.

It is shown that that every finitely generated left congruence is the join of a finitely generated trace minimal left congruence and a finitely generated idempotent separating left congruence.  Characterisations of inverse semigroups that are left Noetherian, or such that the universal congruence is finitely generated are given.
\end{abstract}

\section{Introduction}

There is a rich history of the study of congruence lattices for inverse semigroups.  In broad generality there are two ways to describe a congruence.  The first uses the fact that a congruence is completely determined by the equivalence classes of idempotents, and thus describes such collections of sets; the second makes use of the fact that a congruence is determined by its trace, the restriction to the idempotents, and its kernel, the union of the congruence classes containing idempotents.

Following the former philosophy, in 1974 Meakin \cite{Meakin} developed left-kernel systems which are exactly the collections of subsets that are the equivalence classes of a left congruence containing idempotents.  As well, he showed that every congruence on the semilattice of idempotents arises as the restriction of a left congruence on the whole semigroup, and that the set of left congruences that agree on the idempotents forms an interval in the lattice of left congruences.  Also, a description of the maximum and minimum elements in this interval is given.

Following the second philosophy, in 1992 Petrich and Rankin \cite{PetrichRankin} described the kernel-trace approach to one-sided congruences on inverse semigroups.  From a lattice perspective the kernel-trace approach is advantageous as the ordering of left congruences is induced from the natural orderings on the sets of kernels and traces.  Petrich and Rankin show that the function taking a left congruence to its trace is a $\cap$-homomorphism but not a $\vee$-homomorphism, and that the corresponding function for the kernel is neither.  They also show that the set of left congruences with a given kernel has a maximum element, but in general no minimum element.

In this paper we introduce and characterise the {\em inverse kernel} for left a congruence and give necessary and sufficient conditions for an inverse subsemigroup and a congruence on the idempotents to form an inverse congruence pair (that is the trace and inverse kernel of a left congruence).  We show that the set of congruences that share an inverse kernel has a minimum element.  Also, we describe the sublattice of left congruences having a fixed trace.  The function taking a left congruence to its inverse kernel is shown to be a $\cap$-homomorphism, and a description of the join of two left congruences in terms of the trace and inverse kernel is given.  We apply the inverse kernel description to the bicyclic monoid to recover the lattice of left congruences (see Nico \cite{nico}, and Duchamp \cite{duchamp1986etude}). Finally we consider left congruences in terms of generating sets.  It is shown that we can choose generating sets that exactly correspond to the trace and inverse kernel, and we determine a generating set for the trace in terms of the generating set for the initial left congruence.  This leads to a discussion of finitely generated left congruences, and it is shown that the standard description of left Noetherian semigroups \cite{kozhukhov1980semigroups} is a natural consequence of the inverse kernel trace description for left congruences.

\section{Preliminaries}

We assume a familiarity with standard results in inverse semigroup theory (\cite{Howie} Chapter 5).  Unless otherwise stated $S$ will denote to an inverse semigroup, and $E=E(S)$ will be the semilattice of idempotents of $S.$  Greek letters denote equivalence relations, and $[a]$ is the equivalence class of $a,$ or $[a]_\rho$ if $\rho$ is not clear.

The set of left congruences on $S$ can be ordered by inclusion; we denote this lattice by \lcs/.  It is straightforward that the identity congruence $\iota,$ and the universal congruence $\omega,$ are minimum and maximum elements respectively in \lcs/.  Also write \rcs/ for the lattice of right congruences on $S$.   
A subsemigroup of $S$ is said to be {\em full} if it contains all idempotents of $S.$  The set of all full subsemigroups of $S$ forms a lattice under set inclusion; which is denoted $\mathfrak{K}(S).$  The set $\mathfrak{V}(S)$ of full inverse subsemigroups of $S$ forms a subset of $\mathfrak{K}(S).$

Given a binary relation $R\subseteq S\times S,$ it is often of interest to consider the smallest left congruence containing $R.$  We will denote this by $\langle R\rangle.$  We recall that for $R\subseteq S\times S,$ we have $a\ \langle R\rangle\ b$ if and only if $a=b$ or there are sequences $\{ (x_i,y_i)\ |\ i=1,\dots,n \}\subseteq R,$ and $\{c_i\ |\ i=1,\dots,n \}\subseteq S $ with 
$$a=c_1x_1,\ c_ix_i=c_{i+1}y_{i+1} \text{ for } i=1,\dots,n-1,\ c_ny_n=b.$$ 
We call such a pair of sequences an $R$ sequence from $a$ to $b$.  Again for a binary relation $R$ on $S$ and a subset $T\subseteq S,$ write $R|_T= R\cap (T\times T);$ the restriction of $R$ to $T.$ 

\begin{deff}\label{deff1}
For an equivalence relation $\kappa$ on $S,$ the \textit{kernel} of $\kappa$ is: 
$$\ker/(\kappa)=\bigcup_{e\in E} [e]_\kappa$$ 
and  the \textit{trace} of $\kappa$ is:
$$\tra/(\kappa)=\kappa|_E.$$
\end{deff}

In general this will be applied to a left congruence $\rho$, in which case it is easily verified that the the trace is a congruence on $E.$  Furthermore the kernel is a full subsemigroup of $S,$ however in general it is not an inverse subsemigroup.  The following is the main result of \cite{PetrichRankin}, and says that a congruence is determined by its kernel and trace.  

\begin{theorem}[\cite{PetrichRankin} 3.5] \label{PR1}
Let $\tau$ be a congruence on $E,$ and $K$ a full subsemigroup of $S$ that has the following properties:
\begin{enumerate}[label={(\roman*)}]
\item for all $a\in S,\ b\in K, a\geq b \text{ and } a^{-1}a\ \tau \ b^{-1}b$ imply that $a\in K$;
\item for all $a\in K,\ e,f\in E,\ e\ \tau \ f $ implies that $a^{-1}ea\ \tau \ a^{-1}fa$;
\item for every $a\in K,$ there exists $b\in S$ with $a\geq b,\ a^{-1}a\ \tau \ b^{-1}b$ and $b^{-1}\in K$.
\end{enumerate}
Then define a relation $\rho=\rho_{(K,\tau)}:$
$$a\ \rho\ b\ \iff\ a^{-1}b,b^{-1}a\in K,\ a^{-1}bb^{-1}a\ \tau\ a^{-1}a,\ b^{-1}aa^{-1}b\ \tau\ b^{-1}b.$$
The relation $\rho_{(K,\tau)}$ is a left congruence such that $(K,\tau)=(\ker/\rho_{(K,\tau)},\tra/\rho_{(K,\tau)}).$

Conversely if $\rho$ is a left congruence on $S$ then the pair $(\ker/(\rho),\tra/(\rho))$ obeys $(i),(ii),(iii)$ and $\rho=\rho_{(\ker/(\rho),\tra/(\rho))}.$ 
\end{theorem}

Hence we may realise \lcs/ as a subset of $\mathfrak{K}(S)\times \mathfrak{C}(E),$ as previously mentioned the natural ordering of pairs $(\ker/(\rho),\tra/(\rho))$ (inclusion in both coordinates) coincides with the ordering on left congruences.  Explicitly if $\rho_1,\rho_2$ are left congruences on $S$, then
$$ \rho_1\subseteq \rho_2 \iff \tra/(\rho_1)\subseteq \tra/(\rho_2) \text{ and } \ker/(\rho_1)\subseteq \ker/(\rho_2).$$
Consequently,
$$ \rho_1 = \rho_2 \iff \tra/(\rho_1) = \tra/(\rho_2) \text{ and } \ker/(\rho_1) = \ker/(\rho_2).$$

\begin{deff}
The trace map is the function 
$$\tra/:\mathfrak{LC}(S)\rightarrow \mathfrak{C}(E);\ \rho\mapsto\tra/(\rho),$$
and the kernel map is the function:
$$\ink/:\mathfrak{LC}(S)\rightarrow \mathfrak{K}(S);\ \rho\mapsto\ink/(\rho). $$
If $\rho$ is a left congruence on $S$ then $[\rho]_{\tra/}=\{\rho^\prime\in\mathfrak{LC}(S)\ |\ \tra/(\rho^\prime)=\tau\}$ is the trace class of $\rho.$  Also $[\rho]_{\ink/}=\{\rho^\prime\in\mathfrak{LC}(S)\ |\ \ink/(\rho^\prime)=\ink/(\rho) \}$ is the kernel class of $\rho.$
\end{deff}

It is a straightforward observation from the preceding remarks that the trace and kernel maps are order preserving functions.  It is known that the trace classes are intervals in the lattice of left congruences, however in general kernel classes are less well behaved.  In \cite{PetrichRankin} it is observed that in general they are not intervals.

\begin{deff}\label{deff2}
Let $\tau$ be a congruence on $E$.  Define $N_L(\tau),$ the {\em left-normaliser} of $\tau$ by:
$$N_L(\tau) = \{ a\in S\ |\ e\ \tau\ f \implies a^{-1}ea\ \tau\ a^{-1}fa \}.$$
and $N_R(\tau),$ the {\em right normaliser}:
$$N_R(\tau) = \{ a\in S\ |\ e\ \tau\ f \implies aea^{-1}\ \tau\ afa^{-1} \}.$$
The normaliser $N(\tau)$ is then defined as:
$$N(\tau) = N_R(\tau) \cap N_L(\tau) = \{ a\in S\ |\ e\ \tau\ f \implies aea^{-1}\ \tau\ afa^{-1} \text{ and } a^{-1}ea\ \tau\ a^{-1}fa \}.$$
\end{deff}

It is important to note that $N(\tau)$ is a full inverse subsemigroup, indeed it is the largest inverse subsemigroup contained in $N_L(\tau)$ (or, indeed, in $N_R(\tau)$) in the sense that
$$N(\tau)=\{ a\in N_L(\tau)\ |\ a^{-1}\in N_L(\tau)\}.$$

\begin{theorem}[\cite{Meakin} 3.1, \cite{PetrichRankin} 4.1 \& 4.2]\label{M1}
Let $\tau$ be a congruence on $E$.  Then there exists a left congruence $\rho$ on $S$ with $\tra/(\rho)=\tau$.  Moreover we can define relations $\nu_\tau$, $\mu_\tau$ on $S$: 
\begin{gather*}
\nu_\tau = \{ (a,b)\ |\  a^{-1}a\ \tau\ b^{-1}b, \text{ and }\exists e\in E \text{ such that } e\ \tau\ a^{-1}a \text{ and } ae=be\},\\
\mu_\tau = \{ (a,b)\ |\  a^{-1}bb^{-1}a\ \tau\ a^{-1}a,\ b^{-1}aa^{-1}b\ \tau\ b^{-1}b\ \text{ and } a^{-1}b,\ b^{-1}a\in N_L(\tau) \},
\end{gather*}
and $\nu_\tau,\mu_\tau$ are left congruences on $S$ with trace $\tau$.  Also $\nu_\tau$ is the minimum left congruence with trace $\tau$ and $\mu_\tau$ is the maximum left congruence with trace $\tau.$
\end{theorem}

If we regard $\tau$ as a binary relation on $S,$ then $\langle \tau\rangle$ is the minimum left congruence on $S$ containing $\tau.$   Thus we have a closed form form for $\langle \tau \rangle=\nu_\tau,$ and moreover we have that $\langle\tau\rangle\cap(E\times E)=\tau.$  In particular if $\tau$ is finitely generated then so is $\nu.$  We will continue to use $\nu_\tau$ to denote the minimum left congruence with trace $\tau,$ and will drop the subscript where $\tau$ is clear. 

It is worth noting that for a left congruence $\rho$ on $S$ with trace $\tau,$ $\ker/(\rho)\subseteq N_L(\tau).$  Indeed, suppose that $a\in S,$ $e,f,g\in E,$ with $e\ \tau\ f$ and $a\ \rho\ g;$ then:
$$a^{-1}ea\ \rho\ a^{-1}eg=a^{-1}ge\ \rho\ a^{-1}gf=a^{-1}fg\ \rho\ a^{-1}ga.$$

In \cite{Meakin} Meakin discusses idempotent separating left congruences on $S$ (left congruences with trivial trace), and shows that in this case $\mu$ corresponds to Green's relation $\mathcal{R}.$  As is well known for inverse semigroups this relation is:
$$\mathcal{R}=\{(a,b)\ |\ aa^{-1}=bb^{-1} \}$$
Notice that when $\rho\subseteq \mathcal{R},$ if $a\ \rho\ e$ then $e=aa^{-1},$ whence $a\ \rho\ aa^{-1}$ and thus $a^{-1}a\ \rho\ a^{-1},$ and so $a^{-1}\in \ker/(\rho).$  Hence if $\tra/(\rho)$ is trivial then $\ker/(\rho)$ is an inverse semigroup.  

\begin{theorem}[\cite{Meakin} 4.2]\label{M2}
The restriction of the kernel map to the set of idempotent separating left congruences on $S$ is a lattice isomorphism onto $\mathfrak{V}(S).$
\end{theorem}

Hence for each full inverse subsemigroup $T$ of $S$ there is a unique idempotent separating left congruence with inverse kernel $T.$  We will use $\chi_T$ to denote this left congruence, and will drop the subscript when $T$ is unambiguous.

\section{The Inverse Kernel}

As explained in the previous section it is possible to reconstruct a left congruence from a suitable kernel and trace.  We now show that there is an inverse subsemigroup that we name the inverse kernel contained in the kernel, from which we can also recover the original left congruence.  In \cite{PetrichRankin} it is noted that the primary issue with the kernel is that given $a\in K$ it is not possible to determine to which idempotent $a$ is related.  To this end we make the following definition.

\begin{deff}\label{deff3}
For a left congruence $\rho$ on $S$ the \textit{inverse kernel} of $\rho$ is the set 
$$\ink/(\rho)=\{ a\ |\ a\ \rho\ aa^{-1}\}.$$
\end{deff}

We immediately note that the inverse kernel of a left congruence is contained in the kernel of the left congruence.

\begin{prop}\label{lem1}
Let $\rho$ be a left congruence on $S$ and let $\tau=\tra/(\rho),$ and $K=\ker/(\rho).$  Then the following hold:
\begin{enumerate}[{label=(\roman*)}]
\item $\ink/(\rho)$ is a full inverse subsemigroup of $S;$
\item $\ink/(\rho)=\ker/(\rho \cap \mathcal{R});$
\item $\ink/(\rho)=\{ a\in K\ |\ a^{-1}\in K\};$ \label{ink1}
\item $\ink/(\rho)=K \cap N(\tau).$ \label{lem4}
\end{enumerate}
\end{prop}
\begin{proof}
For the first claim we note that as $e\ \rho\ e=ee^{-1}$ we have that $E\subseteq\ink/(\rho),$ so $\ink/(\rho)$ is full.  Also if $a\ \rho\ aa^{-1}$ then by multiplying on the left by $a^{-1}$ we have $ a^{-1}a\ \rho\ a^{-1},$ hence $\ink/(\rho)$ is inverse.  Suppose that $a,b\in \ink/(\rho),$ so $a\ \rho\ aa^{-1},$ and $b\ \rho\ bb^{-1}.$  Then
$$ab\ \rho\ abb^{-1} = abb^{-1}a^{-1}a\ \rho\ abb^{-1}a^{-1}aa^{-1}=abb^{-1}a^{-1} = (ab)(ab)^{-1}.$$
Thus $\ink/(\rho)$ is a full inverse subsemigroup of $S.$

The second part is immediate, as if $e\in E$ and $a\ (\rho \cap \mathcal{R})\ e$ then $e=aa^{-1}$.

For the third part we have 
$$a\ \rho\ aa^{-1} \implies a^{-1}a\ \rho\ a^{-1} \implies a^{-1}\in K,$$
therefore $\{ a\ |\ a\ \rho\ aa^{-1} \} \subseteq \{ a\in K\ |\ a^{-1}\in K\}.$  To see the reverse inclusion we note that if $a,a^{-1}\in K$ then there are $e,f\in E$ with $a\ \rho\ e$ and $a^{-1}\ \rho\ f,$ also if $a^{-1}\ \rho\ f$ then $fa^{-1}\ \rho\ f\ \rho\ a^{-1}.$
Then we observe:
$$a=aa^{-1}a\ \rho\ aa^{-1}e = eaa^{-1}\ \rho\ eafa^{-1} = afa^{-1}e\ \rho\ afa^{-1}a = af\ \rho\ aa^{-1}.$$

For part $(iv)$ we recall that since $\tra/(\rho)=\tau$ we have $\ink/(\rho)\subseteq \ker/(\rho)\subseteq N_L(\tau).$  Since $N(\tau)=\{a\in N_L(\tau)\ |\ a^{-1}\in N_L(\tau)\}$ part $(iii)$ gives $\ink/(\rho)\subseteq N(\tau)$ and hence $\ink/(\rho)\subseteq{N(\tau)}\cap{K}$.  

For the reverse inclusion suppose $a\in N(\tau)\cap K.$  As $a\in K$ there is $e\in E$ with $a\ \rho\ e.$  Left multiplying this relation by $e$ gives $ea\ \rho\ e\ \rho\ a.$  Left multiplying again, this time by $a^{-1}$ gives $a^{-1}ea\ \rho\ a^{-1}a,$ and as these are both idempotent we have $a^{-1}a\ \tau\ a^{-1}ea.$  Then we may conjugate this relation by $a\in N(\tau)$ to obtain:
$$aa^{-1}= a(a^{-1}a)a^{-1}\ \tau\ a(a^{-1}ea)a^{-1} = aa^{-1}e.$$  
We then observe that
$$a^{-1}=a^{-1}(aa^{-1})\ \rho\ a^{-1}(aa^{-1}e)=a^{-1}e.$$
However, as $a\ \rho\ e$ we also have $a^{-1}a\ \rho\ a^{-1}e.$  Thus we have that $a^{-1}\ \rho\ a^{-1}a$ so $a^{-1}\in \ink/(\rho).$
\end{proof}

From \cref{lem1}$(iii)$ we observe that when the kernel of a left congruence is inverse then the inverse kernel is equal to the kernel.  For instance for two sided congruences and idempotent separating left congruences the notion of kernel and inverse kernel coincide.  For certain classes of inverse semigroups the kernel of a left congruence is always an inverse subsemigroup, including Clifford semigroups (see \cite{PetrichRankin}) and Brandt semigroups (see \cite{PRBrandt}).  The following description of one-sided congruences on inverse semigroups coincides with the kernel-trace description from \cite{PetrichRankin} on these classes of semigroups.

It has been previously mentioned that the normaliser of a congruence on $E$ is the maximum inverse subsemigroup contained in the left normaliser.  From \cref{lem1}$(iii)$ we have that the inverse kernel of a left congruence is the largest inverse subsemigroup contained in the kernel.  Applying this to the maximum left congruence with a fixed trace we have the following corollary.

\begin{corr} \label{inkprop1}\label{lem2}
Let $\tau$ be a congruence on $E$, and let $\mu$ be the maximum left congruence with trace $\tau.$  Then 
$$N(\tau)=\ink/(\mu)= \ker/(\mu\cap\mathcal{R})$$ 
\end{corr}
\begin{proof}
From the description of $\mu$ from \cref{M1} we get that if $a\in N(\tau)$ then $a\ \mu\ aa^{-1}.$  Hence $N(\tau)\subseteq \ink/(\mu).$  However \cref{lem1}\ref{lem4} gives that $\ink/(\mu)\subseteq N(\tau).$ 
\end{proof}

Furthermore from \cref{lem1} we observe that the kernel determines the inverse kernel so the set of left congruences with the same inverse kernel is a union of kernel classes.  We make the following natural definition.

\begin{deff}
The inverse kernel map is the function \lcs/$\rightarrow\mathfrak{V}(S);$ $\rho\mapsto\ink/(\rho).$  For $\rho$ a left congruence the inverse kernel class of $\rho$ is 
$$[\rho]_{\ink/}=\{\kappa\in \mathfrak{LC}(S)\ |\ \ink/(\kappa)=\ink/(\rho)\}.$$
\end{deff}

We recall that for each full inverse subsemigroup there is a unique idempotent separating left congruence for which this subsemigroup is the kernel.  Since the inverse kernel is a full inverse subsemigroup, for each left congruence $\rho$ there is a unique idempotent separating left congruence with the same inverse kernel.  In fact, from \cref{lem1} we know that this is $\rho\cap\mathcal{R},$ since $\rho\cap\mathcal{R}$ is idempotent separating and $\ink/(\rho)=\ker/(\rho\cap\mathcal{R}).$  Later, in \cref{mapsection} we shall discuss this map and it's properties in greater detail.

We also note that if $\rho_1\subseteq \rho_2$ are left congruences then certainly $\ink/(\rho_1)\subseteq \ink/(\rho_2),$ so the inverse kernel map is order preserving.  Suppose $\rho$ is a left congruence and $\chi=\rho\cap\mathcal{R}.$  Then certainly $\ker/(\chi)\subseteq\ker/(\rho)$ and as $\tra/(\chi)$ is trivial we also have $\tra/(\chi)\subseteq\tra/(\rho).$  But then the kernel trace characterisation of left congruences gives that $\chi\subseteq\rho.$  Hence the idempotent separating left congruence is the minimum element in the set of left congruences that share the same inverse kernel. 

In this section we shall show how a left congruence is characterised by its trace and inverse kernel.  

\begin{deff}\label{lcpdeff}
Let $\tau$ be a congruence on $E,$ and let $T\subseteq S$ be a full inverse subsemigroup satisfying the following conditions:
\begin{enumerate}[label={(D\arabic*)}]
\item $T\subseteq N(\tau);$
\item for $x\in S$, if there $e,f\in E$ such that $x^{-1}x\ \tau\ e,\ xx^{-1}\ \tau\ f$ and $xe,fx\in T$ then we have $x\in T.$ \label{lcpdeff2}
\end{enumerate}
Then we say that $(\tau,T)$ is an \textit{inverse congruence pair} for $S.$ 
\end{deff}

For an inverse congruence pair $(\tau,T),$ define the relation:
$$\rho_{(\tau,T)}= \{ (x,y)\ |\ x^{-1}y\in T,\ x^{-1}yy^{-1}x\ \tau\ x^{-1}x,\ y^{-1}xx^{-1}y\ \tau\ y^{-1}y \}.$$

\begin{prop}\label{cor4}
Let $(\tau,T)$ be a inverse congruence pair for $S,$ then $\rho_{(\tau,T)}$ is a left congruence on $S.$  Moreover $\ink/(\rho)=T,$ and $\tra/(\rho)=\tau.$
\end{prop}
\begin{proof}
Let $\rho=\rho_{(\tau,T)}.$ First we show that $\rho$ is a left congruence.  It is immediate that $\rho$ is reflexive and symmetric.  We next show left compatibility.  Suppose that $(a,b)\in \rho.$  Thus $a^{-1}b\in T,$ and writing $a^{-1}c^{-1}ca=f\in E$ we have $a^{-1}c^{-1}c=fa^{-1},$ and
$$(ca)^{-1}(cb)=a^{-1}c^{-1}cb=fa^{-1}b\in T.$$
We note that 
$$(a^{-1}c^{-1}ca)(a^{-1}bb^{-1}a) = a^{-1}c^{-1}cbb^{-1}c^{-1}ca = (ca)^{-1}(cb)(cb^{-1})(ca)$$
Since $a\ \rho\ b$ we have $a^{-1}bb^{-1}a\ \tau\ a^{-1}a,$ thus
$$(ca)^{-1}(cb)(cb^{-1})(ca) = (a^{-1}c^{-1}ca)(a^{-1}bb^{-1}a)\ \tau\ (a^{-1}c^{-1}ca)(a^{-1}a) = (ca)^{-1}(ca)$$
Similarly we obtain $(cb)^{-1}(ca)(ca)^{-1}(cb)\ \tau\ (cb)^{-1}(cb),$ and thus $ca\ \rho\ cb$ and $\rho$ is left compatible. 

We now show that $\rho$ is transitive, to which end suppose that $(a,b),(b,c)\in \rho.$  Thus we have $a^{-1}b,\ b^{-1}c\in T,$ and 
$$a^{-1}bb^{-1}a\ \tau\ a^{-1}a,\ \ b^{-1}aa^{-1}b\ \tau\ b^{-1}b,\ \ b^{-1}cc^{-1}c\ \tau\ b^{-1}b,\ \ c^{-1}bb^{-1}c\ \tau\ c^{-1}c.  $$
We need to show that $c^{-1}a\in T,$ and that $c^{-1}aa^{-1}c\ \tau\ c^{-1}c,$ and $ a^{-1}cc^{-1}a\ \tau\ a^{-1}a.$  

For the latter claim note that as $\tau$ is a congruence:
$$(a^{-1}bb^{-1}a)(a^{-1}cc^{-1}a)\ \tau\ (a^{-1}a)(a^{-1}cc^{-1}a)=a^{-1}cc^{-1}a,$$
also as $T\subseteq N$ and $a^{-1}b\in T$ we conjugate $b^{-1}cc^{-1}b\ \tau\ b^{-1}b$ by $a^{-1}b$ to get
$$(a^{-1}bb^{-1}a)(a^{-1}cc^{-1}a)=(a^{-1}b)(b^{-1}cc^{-1}b)(a^{-1}b)^{-1}\ \tau\ (a^{-1}b)(b^{-1}b)(a^{-1}b)^{-1}= a^{-1}bb^{-1}a$$
We thus obtain:
$$a^{-1}cc^{-1}a\ \tau\ a^{-1}bb^{-1}a\ \tau\ a^{-1}a,$$
and the dual argument gives that $c^{-1}aa^{-1}c\ \tau\ c^{-1}c.$

For the former claim as $\rho$ is left compatible we have that $a^{-1}b\ \rho\ a^{-1}c,\ c^{-1}a\ \rho\ c^{-1}b$ and hence $c^{-1}aa^{-1}b,\ b^{-1}cc^{-1}a\in T.$  We also have $b^{-1}a,\ c^{-1}b\in T,$ and since $T$ is a subsemigroup $(c^{-1}aa^{-1}b)(b^{-1}a),\ (c^{-1}b)(b^{-1}cc^{-1}a)\in T.$  Also 
\begin{gather*}
c^{-1}bb^{-1}c\ \tau\ c^{-1}c\ \tau\ c^{-1}aa^{-1}c=(c^{-1}a)(c^{-1}a)^{-1},\\
a^{-1}bb^{-1}a\ \tau\ a^{-1}a\ \tau\ a^{-1}cc^{-1}a=(c^{-1}a)^{-1}(c^{-1}a).
\end{gather*}
Thus by \ref{lcpdeff2} with $x=c^{-1}a,$ $e=a^{-1}bb^{-1}a,$ and $f=c^{-1}bb^{-1}c$ we have that $c^{-1}a\in T,$ and hence that $\rho$ is a transitive relation and thus is a left congruence on $S.$

Finally we show that $\tra/(\rho)=\tau,$ and $\ink/(\rho)=T.$  Suppose that we have $e,f\in E$ with $e\ \rho\ f,$ then since $e^{-1}=e,$ and $f^{-1}=f$ we have $e\ \tau\ ef\ \tau\ f.$  Hence $e\ \tau\ f.$  Conversely if $e\ \tau\ f$ then it is immediate that $e\ \rho\ f,$ so $\tra/(\rho)=\tau.$  

For the inverse kernel we note that if $a\in T$ then $a\ \rho\ aa^{-1},$ so $T\subseteq \ink/(\rho).$  Moreover if $a\ \rho\ aa^{-1},$ then $a^{-1}(aa^{-1})=a^{-1}\in T$ and since $T$ is inverse we get $a\in T,$ thus $\ink/(\rho)=T.$
\end{proof}

Shortly we shall see that every left congruence is of this form.  However it is beneficial to first consider how it is possible to recover a left congruence from the minimum left congruence with the same trace and the minimum left congruence with the same inverse kernel.

\begin{theorem}\label{thm1}
Let $\rho$ be a left congruence, $\nu$ the minimum left congruence with $\tra/(\nu)=\tra/(\rho)$, and $\chi=\rho\cap\mathcal{R}$.  Then 
$$\rho=\nu\vee\chi.$$
Moreover for $a\in K=\ker/(\rho)$ there is some $f\in E$ such that $f\ \chi\ fa\ \nu\ a.$
\end{theorem}
\begin{proof}
We note that $\nu,\chi\subseteq\rho$, and since $\rho$ is a left congruence we get $\nu\vee\chi\subseteq \rho$, so, as the kernel map is order preserving, certainly $\ker/(\nu\vee\chi)\subseteq\ker/(\rho)$.  Also $\nu \subseteq \nu\vee\chi \subseteq \rho$ and since the trace map is order preserving we certainly have $\text{trace}(\nu\vee\chi)=\text{trace}(\rho)$.  Therefore to complete the proof suffices prove the final claim of the theorem, whence it is immediate that $\ker/(\nu\vee\chi)\supseteq\ker/(\rho)$.

Suppose $a\in\ker/(\rho)$, and let $e$ be an idempotent in the $\rho$-class of $a.$  Let $f=eaa^{-1},$ then from $e\ \rho\ a$ we get 
$$f= aa^{-1}e\ \rho\ aa^{-1}a =a.$$ 
Then as $\rho$ is a left congruence, $fa\ \rho\ f$.  We also note that $(fa)(fa)^{-1}=faa^{-1}f=faa^{-1}=f$, so $fa\ \mathcal{R}\ f$ and therefore $f\ \chi\ fa$.

As $\mu$ is maximum in the trace class of $\rho$ we have $fa\ \mu\ a.$ By \cref{M1} we have $a^{-1}fa\ \tau\ a^{-1}a$, and hence $a^{-1}fa\ \nu\ a^{-1}a$.  As $\nu$ is also a left congruence this gives $fa\ \nu\ a$ completing the proof.
\end{proof}

\begin{theorem}\label{biggy}
Let $(\tau,T)$ be a inverse congruence pair for $S,$ then $\rho_{(\tau,T)}$ is a left congruence on $S$ with trace $\tau$ and inverse kernel $T.$  Conversely, if $\rho$ is a left congruence on $S$ then $(\tra/(\tau),\ink/(\rho))$ is an inverse congruence pair for $S$ and $\rho=\rho_{(\tra/(\tau),\ink/(\rho))}.$
\end{theorem}
\begin{proof}
Suppose that $\rho$ is a left congruence and let $\tau=\tra/(\rho),$ and $T=\ink/(\rho).$  Then by \cref{lem1}\ref{lem4} we have $T=N(\tau)\cap\ker/(\rho)$ so $T\subseteq N(\tau).$  Suppose that $a\in S$ and there are $e,f\in E(S)$ with $ae,fa\in T,$ and $a^{-1}a\ \tau\ e,\ aa^{-1}\ \tau\ f.$  Since $ae,fa\in T$ we know that
$$ae\ \rho\ (ae)(ae)^{-1} = aea^{-1},\ fa\ \rho\ (fa)(fa)^{-1} = faa^{-1}.$$
Since $fa\in T$ we may conjugate $a^{-1}a\ \tau\ e$ by $fa$ to get $faa^{-1}\ \tau\ faea^{-1}.$  As $aa^{-1}\ \tau\ f$ we have $aa^{-1}\ \tau\ aa^{-1}f\ \tau\ f$ and as $a^{-1}a\ \tau e$ we have $ae\ \rho\ aa^{-1}a = a.$  We then observe that:
$$aa^{-1}\ \tau\ faa^{-1}\ \tau\ faea^{-1} = aea^{-1}f\ \tau\ aea^{-1}(aa^{-1}) = aea^{-1}\ \rho\ ae\ \rho\ a.$$  Hence $a\in \ink/(\rho)=T,$ so \ref{lcpdeff2} is satisfied and $(\tau,T)$ is an inverse congruence pair.

It remains to show that $\rho=\rho_{(\tau,T)}.$  We know that $\tra/(\rho_{(\tau,T)})=\tau=\tra/(\rho).$ Let $\nu$ be the minimum left congruence with trace $\tau$, by \cref{thm1} we have that $\rho=(\rho\cap\mathcal{R})\vee\nu,$ and also $\rho_{(\tau,T)}=(\rho_{(\tau,T)}\cap\mathcal{R})\vee\nu.$  However $\ink/(\rho)=T=\ink/(\rho_{(\tau,T)})$ and since idempotent separating congruences are uniquely determined by their inverse kernel $\rho_{(\tau,T)}\cap\mathcal{R}=\rho\cap\mathcal{R}.$  Hence $\rho=\rho_{(\tau,T)}.$
\end{proof}

Thus left congruences on inverse semigroups are determined by their trace and inverse kernel, and thus we may realise the lattice of left congruences as a subset of $\mathfrak{C}(E)\times\mathfrak{V}(S).$  We denote this set $\mathfrak{IP}(S).$ As in the case of the kernel trace description the ordering of left congruences coincides with the natural ordering in the lattice $\mathfrak{C}(S)\times\mathfrak{V}(S).$

\begin{corr}\label{order lem}
Let $\rho_1,\rho_2$ be left congruences on $S$, then
$$ \rho_1\subseteq \rho_2 \iff \tra/(\rho_1)\subseteq \tra/(\rho_2) \text{ and } \ink/(\rho_1)\subseteq \ink/(\rho_2).$$
Consequently,
$$ \rho_1 = \rho_2 \iff \tra/(\rho_1) = \tra/(\rho_2) \text{ and } \ink/(\rho_1) = \ink/(\rho_2).$$
\end{corr}

The following is an important corollary, and is the primary method with which the idea of the inverse kernel trace characterisation of left congruence will be applied in the rest of the paper.

\begin{corr}\label{uniquepair}
Let $\rho$ be a left congruence on $S.$  Let $T=\ink/(\rho),$ and $\tau=\tra/(\rho).$  Then $(\tau,T)$ is the unique pair in $\mathfrak{C}(E)\times\mathfrak{V}(S)$ such that $(\tau,T)$ is an inverse congruence pair, and
$$\rho=\rho_{(\tau,T)}=\chi_T\vee\nu_\tau.$$
\end{corr}

At this juncture it is worth mentioning the connection between the inverse kernel and the natural isomorphism between the lattices of left and right congruence lattices.
$$\rho\mapsto \rho_{-1}=\{(a^{-1},b^{-1})\ |\ (a,b)\in \rho\}$$

We recall the equivalent expression for the inverse kernel of left congruence from \cref{lem1} \ref{ink1} which we can write as:
$$\ink/(\rho) = \{ a\in S\ |\ \exists e,f\in E \text{ such that } a\ \rho\ e,\ a^{-1}\ \rho\ f \}.$$
This is defined for any equivalence relation on $S,$ and when $\rho$ is a right congruence is equivalent to $\{ a\in S\ |\ a\ \rho\ a^{-1}a \}.$  We take this as the definition of an inverse kernel for a right congruence.  We note that \cref{lcpdeff} is self dual with respect to taking the inverse, thus we obtain the following corollary:

\begin{corr}
The pair $(\tau,T)$ is the trace and inverse kernel of a left congruence $\rho$ if and only if it is the trace and inverse kernel of a right congruence.  Moreover if $\rho$ is a left congruence then the right congruence with the same trace and inverse kernel is $\rho_{-1}.$
\end{corr} 

It is also of interest to consider how the kernel and inverse kernel of a left congruence are related.  We have seen that the inverse kernel is the largest inverse subsemigroup contained in the kernel, it is however possible to say more; from the isomorphism between lattices of left and right congruences ($\rho\mapsto\rho_{-1}$) it is immediate that $\tra/(\rho_{-1})=\tra/(\rho)$ and $\ker/(\rho_{-1})=\{a^{-1}\ |\ a\in \ker/(\rho)\}.$  By \cref{lem1} it is then clear that $\ink/(\rho)=\ker/(\rho)\cap\ker/(\rho_{-1})=\ink/(\rho_{-1}).$  Conversely, starting with the inverse kernel of a left congruence, by \cref{thm1} we get that for a left congruence $\rho$ on $S$ with trace $\tau$ the kernel is:
$$\ker/(\rho)=\bigcup_{a\in\ink/(\rho)}[a]_{\nu}.$$

\section{Trace classes}

Motivated by results describing the lattice of idempotent separating left congruences we describe the trace class for an arbitrary trace.  We also give an inverse kernel trace description of the maximum and minimum elements in each trace class.

We define the centraliser of a trace: 
$$C(\tau)= \{ a\in N(\tau)\ |\ \exists e\in E \text{ such that } e\ \tau\ a^{-1}a \text{ and } ae=e\}.$$
We note that this is a full inverse subsemigroup of $S.$  This is straightforward to show either directly or noting that $C(\tau)=N(\tau)\cap\ker/(\nu_\tau)=\ink/(\nu_\tau).$ 

We can also observe that $C(\tau)$ is self conjugate in $N(\tau).$  Indeed, suppose $a\in C(\tau)$ and $b\in N(\tau),$ so there is $e\in E$ with $e\ \tau\ a^{-1}a$ and $ae=e.$  Then $beb^{-1}\in E$ and 
$$(bab^{-1})(beb^{-1})= baeb^{-1} =beb^{-1}$$
We then observe that:
$$beb^{-1}=(beb^{-1})(beb^{-1})= (beb^{-1})(ba^{-1}b^{-1})(bab^{-1})(beb^{-1}) = (ba^{-1}b^{-1})(bab^{-1})(beb^{-1}).$$
Also as $b\in N(\tau)$ we may conjugate $a^{-1}a\ \tau\ e$ to get $ba^{-1}ab^{-1}\ \tau\ beb^{-1}.$  Hence we have:
$$beb^{-1}=(bab^{-1})^{-1}(bab^{-1})(beb^{-1})\ \tau\ (bab^{-1})^{-1}(bab^{-1})(ba^{-1}ab) = (bab^{-1})^{-1}(bab^{-1}).$$
Hence $bab^{-1}\in C(\tau)$ and thus $C(\tau)$ is a self conjugate full inverse subsemigroup of $N(\tau).$

\begin{prop}[see \cite{PetrichRankin}, Proposition 6.4]\label{PR3}
Let $\tau$ be a congruence on $E,$ and let $N=N(\tau).$  Let $\psi = \nu_\tau \cap (N\times N).$   Then $\psi$ is a two-sided congruence on $N$, and 
$$\psi = \{(a,b) \ |\  a^{-1}a\ \tau\ b^{-1}b,\ ab^{-1}\in C(\tau)\} $$ 
Moreover $\psi$ is the minimum congruence on $N$ with trace $\tau.$ 
\end{prop}
\begin{proof}
This can be deduced from \cite{PetrichRankin} Prop 6.4 and the usual kernel-trace description of a two sided congruence on an inverse semigroup, however as we shall rely heavily on this result it is beneficial to include a direct proof.

We seek to prove that $\langle\tau\rangle=\psi.$  Where by $\langle\tau\rangle$ we mean the left congruence on $N(\tau)$ generated by $\tau.$  Suppose $(a,b)\in \psi,$ so $a^{-1}a\ \tau\ b^{-1}b$ and there is $e\in E$ such that $ba^{-1}ab^{-1}\ \tau\ e$ and $ab^{-1}e=e.$  We note that as $b\in N(\tau)$ we may conjugate $ba^{-1}ab^{-1}\ \tau\ e$ by $b$ to get: 
$$a^{-1}ab^{-1}b = b^{-1}ba^{-1}ab^{-1}b\ \tau\ b^{-1}eb.$$
We then get that:
$$a=aa^{-1}a\ \langle\tau\rangle\ aa^{-1}ab^{-1}b\ \langle\tau\rangle\ ab^{-1}eb = eb = bb^{-1}eb\ \langle\tau\rangle\ ba^{-1}ab^{-1}b = ba^{-1}a\ \langle\tau\rangle\ bb^{-1}b = b$$ 
Thus $\psi\subseteq \langle\tau\rangle.$

Next we will show that $\langle\tau\rangle\subseteq\psi.$  We recall that if $a\ \langle\tau\rangle\ b$ then there is a $\langle\tau\rangle$ sequence from $a$ to $b$, so we have $\{c_1,\dots,c_n\}\subseteq N(\tau)$ and $\{(e_1,f_1),\dots,(e_n,f_n)\}\subseteq \tau$ such that
$$a=c_1e_1,\ \text{ for } 1\leq i\leq n-1\ \text{ we have } c_if_i=c_{i+1}e_{i+1},\ \text{ and } c_nf_n=b.$$

We note that then $a^{-1}a=e_1c_1^{-1}c_1$ for $1\leq i\leq n-1$ we have $f_ic_i^{-1}c_i=e_{i+1}c_{i+1}^{-1}c_{i+1}$ and $f_nc_n^{-1}c_n=b^{-1}b.$  Then we observe 
$$a^{-1}a = e_1c_1^{-1}c_1\ \tau\ f_1c_1^{-1}c_1 = e_1c_2^{-1}c_2\ \tau\ \dots\ \tau\ f_nc_n^{-1}c_n = b^{-1}b.$$
Thus $a^{-1}a\ \tau\ b^{-1}b.$  To show $ab^{-1}\in C(\tau)$ we will induct on the length of $\langle\tau\rangle$ sequence.  We note that if the sequence is of length $1$ then there is some $(e,f)\in \tau$ and $c\in N(\tau)$ such that $a=ce,$ and $b=cf.$  Then $ab^{-1}=cefc^{-1},$ and $C(\tau)$ is full and self conjugate we have that $ab^{-1}\in C(\tau).$  

Suppose now that for all $x,y\in N(\tau)$ such that there is a $\langle\tau\rangle$ sequence of length $n-1$ from $x$ to $y$ we have that $xy^{-1}\in C(\tau).$  Suppose there is a $\langle\tau\rangle$ sequence of length $n$ from $a$ to $b.$  Then there is a $\langle\tau\rangle$ sequence of length $n-1$ from $a$ to $c_{n-1}f_{n-1}=c_ne_n,$ thus $ae_nc_n^{-1}\in C(\tau).$  Hence there is some $g\in E$ with $c_ne_na^{-1}ac_n^{-1}\ \tau\ g$ and $ae_nc_n^{-1}g=g.$  Since $c_n\in N(\tau)$ we can conjugate $c_ne_na^{-1}ac_n^{-1}\ \tau\ g$ by $c_n$ to obtain: $c_n^{-1}c_ne_na^{-1}a\ \tau\ c_n^{-1}gc_n.$  We then note that as $f_n\ \tau\ e_n$ we have $f_n\ \tau\ e_nf_n,$ we then note:
$$f_nc_n^{-1}c_na^{-1}a\ \tau\ f_ne_nc_n^{-1}c_na^{-1}a\ \tau\ f_ne_nc_n^{-1}gc_n.$$
We then conjugate this relation by $c_n^{-1}$ to obtain:
$$(af_nc_n^{-1})^{-1}(af_nc_n^{-1})= c_nf_na^{-1}ac_n^{-1}\ \tau\ c_nf_ne_nc_n^{-1}gc_nc_n^{-1} = gc_nf_ne_nc_n^{-1}.$$
Also 
$$af_nc_n^{-1}(gc_nf_ne_nc_n^{-1}) = ae_nc_n^{-1}gc_nf_ne_nc_n^{-1} = gc_n^{-1}f_ne_nc_n^{-1}.$$
Hence $a(c_nf_n)^{-1} = ab^{-1}\in C(\tau).$  Thus $\langle\tau\rangle=\psi,$ and thus $\psi$ is a left congruence on $N(\tau).$

Since we have that $\psi$ is the left congruence on $N(\tau)$ generated by $\tau$ to complete the proof it suffices to show that $\psi$ is a right congruence.  Suppose that $a\ \psi\ b$ and $c\in N(\tau).$  Then we conjugate the relation $a^{-1}a\ \tau\ b^{-1}b$ by $c$ to get
$$(ac)^{-1}(ac)=c^{-1}a^{-1}ac\ \tau\ c^{-1}b^{-1}bc = (bc)^{-1}(bc).$$
Also since $C(\tau)$ is a full subsemigroup and $ab^{-1}\in C(\tau)$ we have that 
$$(ac)(bc)^{-1} = acc^{-1}b^{-1} = ab^{-1}(bcc^{-1}b^{-1})\in C(\tau).$$
Thus $ac\ \tau\ bc$ and so $\psi$ is a right congruence. 
\end{proof}

We note that we have not used any prior knowledge about $\nu_\tau$ in the proof of \cref{PR3}, and in fact we can deduce directly from \cref{PR3} that $C(\tau)=\ink/(\nu_\tau).$  Indeed if $\tau$ is a congruence on $E$ and $\nu=\langle\tau\rangle,$ then if $a\in\ink/(\nu)$ then $a\in N(\tau)$ and thus $a\in \ink/(\nu|_{N(\tau)}).$  As $a\ \nu\ aa^{-1}$ by \cref{PR3} we have that $a^{-1}a\ \tau\ aa^{-1}$ and $a(aa^{-1})\in C(\tau).$  Hence there is some $e\in E$ such that $(aa^{-1})(a^{-1}a)\ \tau\ e$ and $a(aa^{-1})e=e.$  Then we note that 
$$(aa^{-1})e= (aa^{-1})a(aa^{-1})e = a(aa^{-1})e = e.$$
Hence $ae=a(aa^{-1})e= e.$  We also observe that 
$$a^{-1}a\ \tau\ (aa^{-1})(a^{-1}a)\ \tau\ e.$$
Thus $a\in C(\tau).$  Thus we have that $\ink/(\nu)\subseteq C(\tau).$  

Conversely we observe that if $a\in C(\tau)$ then certainly $a\ \nu\ a^{-1}a,$ so $a\in\ker/(\nu).$  However as previously noted for a two sided congruence the notion of kernel and inverse kernel coincide, hence $a\in \ink/(\nu).$  Thus $C(\tau)= \ink/(\nu).$

\begin{lemma}\label{sat lem}
Let $\tau$ be a congruence on $E$ and $T\subseteq S$ be a full inverse subsemigroup.  Then $(\tau,T)$ is an inverse congruence pair if and only if $T$ is a union of $\nu|_{N(\tau)}$ equivalence classes.
\end{lemma}
\begin{proof}
Let $\psi=\nu|_{N(\tau)}.$ Suppose that $(\tau,T)$ is an inverse congruence pair and let $\rho=\rho_{(\tau,T)}.$  Then we have that $T\subseteq N(\tau).$  We want to show that $T$ is saturated by $\psi,$ to which end suppose that $a\in N(\tau)$ and $a\ \psi\ b$ for some $b\in T.$  From the description of $\psi$ from \cref{PR3} we have $a^{-1}a\ \tau\ b^{-1}b,$ and $ab^{-1}\in C(\tau).$  Since $a\in N(\tau)$ we can conjugate $a^{-1}a\ \tau\ b^{-1}b$ by $a$ to get $aa^{-1}\ \tau\ ab^{-1}ba^{-1}.$ Then letting $e=b^{-1}b,\ f=ab^{-1}ba^{-1}$ we have $ae=ab^{-1}b=fa.$  Then we note that as $b,ab^{-1}\in T$ we have that $ab^{-1}b\in T,$ so by \ref{lcpdeff2} we get $a\in T,$ so $T$ is saturated by $\psi.$  

Conversely suppose that $T$ is a union of $\nu|_N$ equivalence classes.  Certainly $T\subseteq N$ by definition.  Suppose that $a\in S$ and there exist $e,f\in E$ such that $a^{-1}a\ \tau\ e,\ aa^{-1}\ \tau\ f,$ and $ae,fa\in T.$  We first show that $ a\in N.$  Suppose that $g,h\in E$ with $g\ \tau\ h.$  Then as $fa\in N(\tau)$ we conjugate $g\ \tau\ h$ by $fa$ to get $faga^{-1}\ \tau\ faha^{-1}.$  Since $aa^{-1}\ \tau\ f$  we then have
$$aga^{-1}= aa^{-1}(aga^{-1})\ \tau\ f(aga^{-1})\ \tau\ f(aha^{-1})\ \tau\ aa^{-1}(aha^{-1}) = aha^{-1}.$$
Conjugating $g\ \tau\ h$ by $ae$ and using that $eaa^{-1}\ \tau\ a^{-1}a$ a similar argument gives that $a^{-1}ga\ \tau\ a^{-1}ha,$ hence $a\in N.$  

We then note that $a^{-1}a\ \tau\ e$ implies that $a\ \nu\ ae.$  As $ae\in T$ and $T$ is saturated by $\nu|_N,$ we have $a\in T.$
\end{proof}
   
We can now the main result of this section, an extension of \cref{M1} to an arbitrary trace class.

\begin{theorem}\label{thm2}
Let $\tau$ be a congruence on $E$, let $N=N(\tau)$, $\nu$ the minimum left congruence on $S$ with trace $\tau$, and $\psi$ be the restriction of $\nu$ to $N.$  Then the lattice of left congruences on $S$ with trace $\tau$ is isomorphic to the lattice of full inverse subsemigroups of $N/\psi$
\end{theorem}
\begin{proof}
We know that the set of full inverse subsemigroups of $N$ that form an inverse congruence pair with $\tau$ consists precisely of those saturated by $\psi.$  Therefore to complete the proof it suffices to note that by standard universal arguments this set is exactly the pre-image under the natural homomorphism $N\mapsto N/\psi$ of the set of full inverse subsemigroups of $N/\psi.$
\end{proof}

Descriptions of the maximum and minimum elements in a trace interval are available in both \cite{Meakin} and \cite{PetrichRankin}.  Let $\tau$ be a congruence on $E,$ we have noted that the inverse kernel of the minimum left congruence with trace $\tau$ is $C(\tau).$  We also note that $(\tau,N(\tau))$ is certainly an inverse congruence pair.  In terms of the inverse kernel trace description we get the following description for the maximum and minimum elements in a trace class.

\begin{corr}
Let $\tau$ be a congruence on $E(S).$  The minimum and maximum left congruences with trace $\tau$ are respectively:
\begin{gather*}
\rho_{(\tau,C(\tau))} = \{ (x,y)\ |\ x^{-1}y\in C(\tau),\ x^{-1}yy^{-1}x\ \tau\ x^{-1}x,\ y^{-1}xx^{-1}y\ \tau\ y^{-1}y \}\\
\rho_{(\tau,N(\tau))} = \{ (x,y)\ |\ x^{-1}y\in N(\tau),\ x^{-1}yy^{-1}x\ \tau\ x^{-1}x,\ y^{-1}xx^{-1}y\ \tau\ y^{-1}y \}
\end{gather*}

\end{corr}

\section{Inverse Kernel Classes}

We now seek to give analogues for inverse kernel classes of results concerning trace classes.  We know that given a full inverse subsemigroup there is an idempotent separating left congruence such that this subsemigroup is the inverse kernel and that this is the minimum element in the inverse kernel class.  Unfortunately though in general there is no maximum element in an inverse kernel class. 

\begin{figure}[!ht]
\centering
\begin{minipage}{.5\textwidth}
  \centering
\begin{tikzpicture}[scale=.5,-,auto,thick,main node/.style={circle},normal node/.style={circle,fill=black,inner sep=0pt,minimum size=5pt}]

\node[normal node] (a1) at (7,18){};
\node[normal node] (a2) at (6,17){};
\node[normal node] (a3) at (8,17){};
\node[normal node] (a4) at (7,16){};

\node[normal node] (b1) at (1,13){};
\node[normal node] (b2) at (0,12){};
\node[normal node] (b3) at (2,12){};
\node[normal node] (b4) at (1,11){};

\node[normal node] (c1) at (7,13){};
\node[normal node] (c2) at (6,12){};
\node[normal node] (c3) at (8,12){};
\node[normal node] (c4) at (7,11){};

\node[normal node] (d1) at (13,13){};
\node[normal node] (d2) at (12,12){};
\node[normal node] (d3) at (14,12){};
\node[normal node] (d4) at (13,11){};

\node[normal node] (e1) at (4,8){};
\node[normal node] (e2) at (3,7){};
\node[normal node] (e3) at (5,7){};
\node[normal node] (e4) at (4,6){};

\node[normal node] (f1) at (10,8){};
\node[normal node] (f2) at (9,7){};
\node[normal node] (f3) at (11,7){};
\node[normal node] (f4) at (10,6){};

\node[normal node,label={\small $I_{1,2}$}] (g1) at (7,3){};
\node[normal node,label={\small $I_{1}$}] (g2) at (6,2){};
\node[normal node,label={\small $I_{2}$}] (g3) at (8,2){};
\node[normal node,label={\small $I_{\emptyset}$}] (g4) at (7,1){};

\draw[rotate around={45:(9.5,6.5)}] (9.5,6.5) ellipse (15pt and 30pt);
\draw[rotate around={-45:(4.5,6.5)}] (4.5,6.5) ellipse (15pt and 30pt);
\draw[rotate around={45:(13.5,12.5)}] (13.5,12.5) ellipse (15pt and 30pt);
\draw[rotate around={45:(12.5,11.5)}] (12.5,11.5) ellipse (15pt and 30pt);
\draw[rotate around={-45:(0.5,12.5)}] (0.5,12.5) ellipse (15pt and 30pt);
\draw[rotate around={-45:(1.5,11.5)}] (1.5,11.5) ellipse (15pt and 30pt);
\draw[rounded corners] (7,18.5)--(8.5,17)--(7,15.5)--(5.5,17)--cycle;
\draw[rounded corners] (5.5,12)--(6,12.5)--(8,12.5)--(8.5,12)--(7,10.5)--cycle;

\draw (6.2,3.2)--(4.6,5.7);
\draw (7.8,3.2)--(9.4,5.7);

\draw (4.6,8.2)--(6.4,10.7);
\draw (9.4,8.2)--(7.6,10.7);
\draw (3.2,8.2)--(1.6,10.7);
\draw (10.8,8.2)--(12.4,10.7);

\draw (1.8,13)--(6,16.1);
\draw (12.2,13)--(8,16.1);
\draw (7,13.5)--(7,15.3);

\end{tikzpicture}
\caption{$\mathfrak{C}(E(\mathcal{I}_2))$}
  \label{congs}
\end{minipage}%
\begin{minipage}{.5\textwidth}
  \centering
  \begin{tikzpicture}[scale=.5,-,auto,thick,main node/.style={circle},normal node/.style={circle,fill=none,inner sep=0pt,minimum size=20pt}]

\node[normal node] (a1) at (1,5){$\mathcal{I}_2$};
\node[normal node] (a2) at (1,3){$T$};
\node[normal node] (a3) at (1,1){$E$};

\path[every node/.style={blue,font=\sffamily\small}]

(a1) edge (a2)
(a2) edge (a3);

\end{tikzpicture}
\caption{$\mathfrak{V}(\mathcal{I}_2)$}
  \label{subs}
\end{minipage}

\begin{minipage}{.45\textwidth}
  \centering
\begin{tikzpicture}[scale=.5,-,auto,thick,main node/.style={circle,double,draw=black,fill=black,inner sep=0pt,minimum size=5pt},normal node/.style={circle,fill=black,inner sep=0pt,minimum size=3pt}]

\node[main node] (a1) at (5,15){};
\node[normal node] (a2) at (5,14){};
\node[normal node] (a3) at (5,13){};

\node[normal node] (b1) at (0,11){};
\node[normal node] (b2) at (0,10){};
\node[main node] (b3) at (0,9){};

\node[main node] (c1) at (5,11){};
\node[main node] (c2) at (5,10){};
\node[normal node] (c3) at (5,9){};

\node[normal node] (d1) at (10,11){};
\node[normal node] (d2) at (10,10){};
\node[main node] (d3) at (10,9){};

\node[normal node] (e1) at (2,7){};
\node[normal node] (e2) at (2,6){};
\node[main node] (e3) at (2,5){};

\node[normal node] (f1) at (8,7){};
\node[normal node] (f2) at (8,6){};
\node[main node] (f3) at (8,5){};

\node[main node] (g1) at (5,3){};
\node[main node] (g2) at (5,2){};
\node[main node] (g3) at (5,1){};

\path[every node/.style={blue,font=\sffamily\small}]

(g1) edge (g2)
(g2) edge (g3)

(g3) edge (f3)
(g3) edge (e3)

(g2) edge[bend right] (c2)
(g1) edge[bend left] (c1)

(f3) edge (d3)
(e3) edge (b3)

(f3) edge (c2)
(e3) edge (c2)

(c2) edge (c1)

(c1) edge[bend right] (a1)
(b3) edge (a1)
(d3) edge (a1);

\end{tikzpicture}
\caption{$\mathfrak{LC}(\mathcal{I}_2)\subseteq \mathfrak{C}(E)\times \mathfrak{V}(S)$}
\label{lcons}
\end{minipage}
\begin{minipage}{.45\textwidth}
  \centering
\begin{tikzpicture}[scale=.5,-,auto,thick,main node/.style={circle,double,draw=black,fill=black,inner sep=0pt,minimum size=5pt},normal node/.style={circle,fill=black,inner sep=0pt,minimum size=3pt}]

\node[main node] (a1) at   (5,15){};
\node[normal node] (b1) at (1,14){};
\node[main node] (c1) at   (5,14){};
\node[normal node] (d1) at (9,14){};
\node[normal node] (e1) at (3,13){};
\node[normal node] (f1) at (7,13){};
\node[main node] (g1) at   (5,12){};

\node[normal node] (a2) at (5,9){};
\node[normal node] (b2) at (1,8){};
\node[main node] (c2) at   (5,8){};
\node[normal node] (d2) at (9,8){};
\node[normal node] (e2) at (3,7){};
\node[normal node] (f2) at (7,7){};
\node[main node] (g2) at   (5,6){};

\node[normal node] (a3) at (5,3){};
\node[main node] (b3) at   (1,2){};
\node[normal node] (c3) at (5,2){};
\node[main node] (d3) at   (9,2){};
\node[main node] (e3) at   (3,1){};
\node[main node] (f3) at   (7,1){};
\node[main node] (g3) at   (5,0){};

\path[every node/.style={blue,font=\sffamily\small}]

(g1) edge[bend left] (g2)
(g2) edge[bend left] (g3)

(g3) edge (f3)
(g3) edge (e3)

(g2) edge (c2)
(g1) edge (c1)

(f3) edge (d3)
(e3) edge (b3)

(f3) edge (c2)
(e3) edge (c2)

(c2) edge[bend left] (c1)

(c1) edge (a1)
(b3) edge (a1)
(d3) edge (a1);

\end{tikzpicture}
\caption{$\mathfrak{LC}(\mathcal{I}_2)\subseteq \mathfrak{V}(S)\times \mathfrak{C}(E)$}
\label{lci22}
\end{minipage}
\end{figure}
To illustrate this we give a simple example of the use of the inverse kernel trace description of the lattice of left congruences.  We consider $\mathcal{I}_2$ the symmetric inverse monoid on a $2$ element set.  We label the elements of $\mathcal{I}_n$ by $I_{1,2},I_1,I_2,I_\emptyset,\alpha,\beta,\beta^{-1}$ where $I_X$ is the idempotent with domain $X\subseteq \{1,2\},$ $\alpha$ is the non identity invertible element, and $\beta$ has domain $\{1\}$ and image $\{2\}.$  $\mathcal{I}_2$ has $3$ distinct full inverse subsemigroups: $E=E(\mathcal{I}_2), \mathcal{I}_2$ and $T=E\cup\{\beta,\beta^{-1}\}.$  The semilattice of idempotents is isomorphic to the powerset of a $2$ element set under intersection.  The lattice of congruences on the idempotents is illustrated in \cref{congs}, in which the partitions of the idempotents are shown.  The lattice of full inverse subsemigroups is displayed in\cref{subs}.

The lattice of left congruences is then realised as a subset of the direct product of these two lattices.  After elementary calculations to determine which pairs are inverse congruence pairs we obtain the lattice of left congruences as shown in \cref{lcons} and \cref{lci22}, which both show the direct product $\mathfrak{C}(E)\times\mathfrak{V}(\mathcal{I}_2)$ with the inverse congruence pairs indicated by circled vertices. \cref{lcons} shows the left congruences grouped into trace classes, and \cref{lci22} shows them grouped by inverse kernel.  It is then easy to observe that the inverse kernel class of $E$ contains no maximum elements.

\begin{lemma}\label{normlem}
Let $\{\tau_i\ |\ i\in I\}$ be a set of congruences on $E$ with normalisers $N_i.$ respectively.  Then
\begin{gather*}
N(\bigcap_{i\in I}\tau_i) \supseteq \bigcap_{i\in I} N_i;\\
N(\bigvee_{i\in I}\tau_i)\supseteq \bigcap_{i\in I} N_i.
\end{gather*}
\end{lemma}
\begin{proof}
The first part is straightforward: suppose $a\in \bigcap_{i\in I} N_i$ and $ e\ (\bigcap_{i\in I} \tau_i)\ f.$  Then we have $aea^{-1}\ \tau_i\ afa^{-1}$ for each $i\in I,$ so $aea^{-1}\ (\bigcap_{i\in I} \tau_i)\ afa^{-1}.$  Similarly we obtain $a^{-1}ea\ (\bigcap_{i\in I} \tau_i)\ a^{-1}fa,$ so $a\in N(\bigcap_{i\in I} \tau_i).$

For the second claim we suppose that $e\ (\bigvee_{i\in I} \tau_i)\ f$ and $a\in \bigcap_{i\in I} N_i.$  As $e\ (\bigvee_{i\in I} \tau_i)\ f$ there is some sequence $\{g_1,\dots,g_m\}\subseteq E$ such that
$$ e\ \tau_{i_1}\ g_1\ \tau_{i_2}\ g_2\ \dots\ g_{m-1}\ \tau_{i_{m-1}}\ g_m\ \tau_{i_m}\ f.$$

At each stage in the chain we can conjugate by $a$ and thus obtain: 
$$aea^{-1}\ (\bigvee_{i\in I} \tau_i)\ afa^{-1}.$$  
Then by symmetry in $a,a^{-1}$ we obtain $a\in N(\bigvee_{i\in I} \tau_i).$
\end{proof}

We now give a description of the set of congruences on $E$ that are the traces of an inverse kernel class.  The following is a rewording of \cref{lcpdeff}, in which $T$ is fixed.

\begin{corr}\label{inklat}
Let $T$ be a full inverse subsemigroup of $S.$  For a congruence $\tau$ on $E=E(T)=E(S)$ we have that $(\tau,T)$ is an inverse congruence pair if and only if $\tau$ is normal in $T$ and for each $a\in S\backslash T$ and $e\in E$ with $ae\in T,$ if $a^{-1}a\ \tau\ e$ then $aa^{-1}\ \not\!\tau\ aea^{-1}.$  
\end{corr}

A partial order is said to have the descending chain condition if it contains no infinite descending chains.  Any partially ordered set with the descending chain condition has minimal elements; if a meet-semilattice has the descending chain condition then we note that it contains a minimum element.  Let $\tau$ be a congruence on a semilattice $E.$  Since a congruence class is a subsemilattice we note that if a semilattice has the descending chain condition then each $\tau$-class has a minimum element.  In particular when $E$ is finite $E$ certainly has the descending chain condition.  We observe that when $E$ has the descending chain condition then the usual partial order on $S$ also has the descending chain condition .

\begin{deff}
Let $T\subseteq S$ be a full inverse subsemigroup.  Say $a\in S\backslash T$ is minimal (with respect to T) if $\{ ae\ |\ e\in E\} \subseteq T\cup \{a\}.$  Equivalently $a\in S\backslash T$ is minimal if $b\leq a$ implies $b=a$ or $b\in T.$
\end{deff}

\begin{lemma}\label{dcclem}
Let $S$ be an inverse semigroup such that $E$ has the descending chain condition, and let $T\subseteq S$ be a full inverse subsemigroup.  For each $a\in S\backslash T$ there exists $b\in S\backslash T$ such that $b\leq a$ and $b$ is minimal with respect to $T.$ 
\end{lemma}
\begin{proof}
Suppose $a\in S\backslash T,$ with $a$ not minimal with respect to $T.$  Let $X=\{ae\ |\ e\in E,\ ae\notin T\}.$  Since the partial order on $S$ has the descending chain condition it is clear that $X$ also has the descending chain condition.  Let $g\in E$ be such that $ag$ is minimal in $X.$  Let $b=ag,$ then $b\leq a$ and $b\notin T.$  Also for any $e\in E$ we have $be=age \leq ag,$ so either $be=ag=b,$ or $be=age<ag$ whence $be\in T.$  Hence $b$ is minimal with respect to $T.$
\end{proof}

\begin{prop}
Let $S$ be a inverse semigroup with the descending chain condition, and let $T\subseteq S$ be a full inverse subsemigroup.  If $\tau$ is a congruence on $E(S)$ with $T\subseteq N(\tau),$ then $(\tau,T)$ is an inverse congruence pair if and only if for each $a\in S\backslash T$ with $a$ minimal at least one of $aa^{-1},\ a^{-1}a$ is the minimum in its $\tau$-equivalence class.
\end{prop}
\begin{proof}
Suppose that $a\in S\backslash T$ is minimal, then we immediately obtain that $a^{-1}$ is also minimal.  We also observe that if $e< a^{-1}a$ then $ae< a.$  We initially assume that $(\tau,T)$ is an inverse congruence pair.  Suppose that both $a^{-1}a,\ aa^{-1}$ are not minimum in their $\tau$-class.  Then have $e< a^{-1}a,\ f< aa^{-1},$ with $e\ \tau\ a^{-1}a,\ f\ \tau\ aa^{-1}.$  Then as $a,a^{-1}$ are minimal we get $ae,a^{-1}f\in T.$  But since $(\tau,T)$ is an inverse congruence pair \cref{lcpdeff} \ref{lcpdeff2} gives that $a\in T$ which is a contradiction.

Conversely suppose that the latter condition holds.  We need to verify \ref{lcpdeff2}.  Let $a\in S$ and suppose there are $e,f\in E$ such that $e\ \tau\ a^{-1}a,\ f\ \tau\ aa^{-1}$ and $ae,fa\in T.$  Suppose that $a\notin T;$ we aim for a contradiction.  

Since $S$ satisfies the descending chain condition by \cref{dcclem} we have that there is some $h\in E$ such that $b=ah$ and $b$ is minimal with respect to $T.$  By assumption at least one of $bb^{-1}$ or $b^{-1}b$ is it's $\tau$-class minimum.  If $b^{-1}b$ is a minimum, then 
$b^{-1}b=b^{-1}ba^{-1}a\ \tau\ b^{-1}be$ so that $b^{-1}b=b^{-1}be$ and $b=be.$  Then $b=be=aeh\in T,$ a contradiction.  Similarly $bb^{-1}$ cannot be minimum it its $\tau$-class.  It follows that $a\in T$ and \ref{lcpdeff2} holds.
\end{proof}

\section{Trace and Inverse Kernel maps}\label{mapsection}

We now consider the lattice of left congruences on $S.$  We show that like the trace map and unlike the kernel map, the inverse kernel map is a $\cap$-homomorphism and shall describe the meets and joins of left congruences in terms of the trace and inverse kernel.  

Recall that given a trace $\tau$ we have that $\nu_\tau=\langle\tau\rangle.$  Therefore given traces $\tau_1,\tau_2$ we have $\nu_{\tau_1}\vee\nu_{\tau_2}=\langle \tau_1\vee\tau_2\rangle.$  

\begin{corr}
The map $\tau\mapsto\nu_\tau$ is a lattice embedding $\mathfrak{C}(E)\hookrightarrow\mathfrak{LC}(S).$
\end{corr}

The image of this map is $\{\nu_\tau\ |\ \tau\in \mathfrak{C}(E)\},$ and we refer to this set as the set of trace minimal left congruences.  

It is shown in \cite{PetrichRankin} that the map $\rho\mapsto\tra/(\rho)$ is a $\cap$-homomorphism, so in particular:
$$ \tra/(\rho_1\cap\rho_2)=\tra/(\rho_1)\cap \tra/(\rho_2).$$
We have noted that the kernel map is not in general a $\cap$-homomorphism.  However it is elementary that the inverse kernel map is such.  

\begin{prop}\label{latprop1}
Let $\rho_1=\rho_{(\tau_1,T_1)}$, $\rho_2=\rho_{(\tau_2,T_2)}$ be left congruences on $S.$  Then
$$\ink/(\rho_1\cap\rho_2)=T_1\cap T_2.$$
\begin{proof}
We apply \cref{lem1} to get that $T_i=\{a\ |\ a\ \rho_i\ aa^{-1}\}$  
\begin{align*}
a\in T_1 \text{ and } T_2 & \iff a\ \rho_1\ aa^{-1} \text{ and } a\ \rho_2\ aa^{-1} \\
& \iff a\ (\rho_1\cap\rho_2)\ aa^{-1} \\
& \iff a\in \ink/(\rho_1\cap\rho_2).
\end{align*}
\end{proof}
\end{prop}

Since we know that the restriction of the inverse kernel map to the set of idempotent separating left congruences is onto, and the restriction of the trace map to the set of trace minimal left congruences is onto we have shown the following.

\begin{corr}
The trace and kernel maps are onto meet homomorphisms.
\end{corr}

It is then straightforward to determine the trace and inverse kernel of the meet of left congruences.

\begin{corr}
Let $\rho_1=\rho_{(\tau_1,T_1)}$ and $\rho_2=\rho_{(\tau_2,T_2)}$ be left congruences on $S.$  Then
$$\rho_1\cap\rho_2=\rho_{(\tau_1\cap\tau_2,T_1\cap T_2)}.$$
\end{corr}

It is a non trivial question to determine the join of two left congruences on $S.$ We now show that using the inverse kernel approach provides a mechanism to handle this problem smoothly.

\begin{lemma}\label{joinlem}
Let $\tau$ be a congruence on $E,$ and $\psi=\nu_\tau\cap(N(\tau)\times N(\tau)).$  Let $T\subseteq N(\tau)$ be a full inverse subsemigroup.  Let $V=\bigcup_{t\in T} [t]_\psi.$  Then $(\tau,V)$ is an inverse congruence pair.
\begin{proof}
We first show that $V$ is a full inverse subsemigroup.  Recall from \cref{PR3} that $\psi$ is a two sided congruence on $N(\tau)$.  First we observe that if $a,b\in V$ there exist $x,y\in T$ with $a\ \psi\ x,\ b\ \psi\ y.$  As $\psi$ is a two sided congruence we obtain $ab\ \psi\ xy,$ hence $ab\in V.$  Again as $\psi$ is two sided if $a\ \psi\ b$ then $a^{-1}\ \psi\ b^{-1},$ so as $T$ is inverse it follows that $V$ is inverse.  We also note that as $E\subseteq T\subseteq V$ it is immediate that $V$ is full.  Thus $V$ is a full inverse subsemigroup.

By definition $V$ is saturated by $\psi,$ so applying \cref{sat lem} we have that $(\tau,V)$ is an inverse congruence pair.
\end{proof}
\end{lemma}

\begin{prop}\label{latprop2}
Let $\rho_1=\nu_{\tau_1}\vee\chi_{T_1}$, $\rho_2=\nu_{\tau_2}\vee\chi_{T_2}$ be left congruences on $S.$  Let $\xi$ be the least congruence on $E$ such that $\xi\supseteq\tau_1\vee\tau_2,$ and $N(\xi)\supseteq T_1\vee T_2.$  Let $\psi=\nu_\xi\cap(N(\xi)\times N(\xi)),$ and let $V=\bigcup_{t\in T_1\vee T_2} [t]_\psi.$  Then 
$$\rho_1\vee\rho_2= \rho_{(\xi,V)}.$$
\end{prop}
\begin{proof}
We first note that $\xi$ is well defined as given a family $\{\xi_i\ |\ i\in I\}\subseteq \mathcal{C}(E)$ with $\tau_1,\tau_2\subseteq \xi_i$ and $T_1\vee T_2\subseteq N(\xi_i)$ we may take $\xi=\bigcap_{i\in I} \xi_i,$ and it is immediate that $\tau_1,\tau_2\subseteq \xi,$ and by \cref{normlem} $T_1\vee T_2\subseteq \bigcap_{i\in I} N(\xi_i)\subseteq N(\xi).$  Hence if we take $\{\xi_i\ |\ i\in I\}$ to be set of all congruences on $E$ with these properties, then $\xi$ is well defined and is the smallest congruence on $E$ such that the properties hold.  From \cref{joinlem} we observe that $(\xi,V)$ is a inverse congruence pair.  Let $\rho=\rho_{(\xi,V)},$ then by appeal to \cref{order lem} we have that $\rho_1\vee\rho_2\subseteq \rho.$

We now show that $\rho\subseteq\rho_1\vee\rho_2.$  Let $(\zeta,W)$ be such that $\rho_1\vee\rho_2=\rho_{(\zeta,W)}.$   From \cref{order lem} we must have that $T_1\vee T_2\subseteq W\subseteq V$ and $\tau_1\vee\tau_2 \subseteq \zeta\subseteq \xi.$  As $(\zeta,W)$ is a inverse congruence pair, we get that $N(\zeta)\supseteq T_1\vee T_2.$  But by definition $\xi$ is the least congruence on $E$ that has these properties.  Therefore we get $\xi\subseteq \zeta,$ and thus $\xi=\zeta.$

Finally we can note that $T_1\vee T_2\subseteq W,$ and $W$ is saturated by $\psi.$  It is then clear that $V\subseteq W$ so $V=W,$ and thus $\rho= \rho_1\vee\rho_2.$
\end{proof}

We know that the $\mathfrak{LC}(S)\cong\mathfrak{IP}(S)\subseteq \mathfrak{C}(E)\times\mathfrak{V}(S).$  We have considered the trace and inverse kernel maps, which map \lcs/ onto the components of the direct product.  We can naturally combine these maps and obtain the function:
$$\Phi: \mathfrak{LC}(S)\rightarrow\mathfrak{C}(E)\times\mathfrak{V}(S);\ \rho\mapsto (\tra/(\rho)\ink/(\rho)).$$
We recall that there are natural lattice embeddings $\mathfrak{V}(S)\rightarrow \mathfrak{LC}(S);\ T\mapsto \chi_T,$ and $\mathfrak{C}(E)\rightarrow \mathfrak{LC}(S);\ \tau\mapsto \nu_\tau.$  We consider the function  
$$\Theta: \mathfrak{C}(E)\times\mathfrak{V}(S)\rightarrow \mathfrak{LC}(S);\ (\tau,T)\mapsto \nu_\tau\vee\chi_T.$$

\begin{theorem}\label{embed thm}
The function $\Phi$ is an meet-homomorphism, and $\Theta$ is an onto join-homomorphism.  Moreover $\Phi\Theta: \mathfrak{LC}(S)\rightarrow\mathfrak{LC}(S)$ is the identity map.
\begin{proof}
That $\Phi$ is a meet-homomorphism is immediate as the trace and inverse kernel maps are meet-homomorphisms.  Suppose $(\tau_1,T_1),(\tau_2,T_2)\in \mathfrak{C}(E)\times\mathfrak{V}(S).$  Then utilising that the trace minimal elements and the idempotent separating left congruences are sublattices we obtain
\begin{align*}
(\tau_1,T_1)\Theta\vee (\tau_2,T_2)\Theta & = (\nu_{\tau_1}\vee\chi_{T_1})\vee(\nu_{\tau_2}\vee\chi_{T_2})\\
& = (\nu_{\tau_1}\vee\nu_{\tau_2})\vee(\chi_{T_1}\vee\chi_{T_2})\\
& = \nu_{\tau_1\vee\tau_2}\vee \chi_{T_1\vee T_2}\\
& = (\tau_1\vee\tau_2,T_1\vee T_2)\Theta
\end{align*}
Thus $\Theta$ is a join-homomorphism.  From \cref{uniquepair} we know that a left congruence $\rho_{(\tau,T)}=\nu_\tau\vee\chi_T.$  Thus it is clear both that $\Phi$ is onto, and that the function $\Phi\Theta$ is the identity map.
\end{proof}
\end{theorem}

We remark that neither is $\Phi$ an join homomorphism, nor is $\Theta$ a meet homomorphism.  To see that $\Phi$ does not preserve join recall that for $\mathcal{I}_2$ there are distinct left congruences with inverse kernel equal to $E$ and join equal to the universal congruence, which has inverse kernel $\mathcal{I}_2.$  

The example of $S=\mathcal{I}_2$ also suffices to show that $\Theta$ is not a meet homomorphism.  Indeed let $\tau_1,\tau_2$ be congruences on $E$ with $N(\tau_1)=E=N(\tau_2),$ $\tau_1\not\leq\tau_2,$ and $\tau_2\not\leq\tau_1.$  Then $\nu_{\tau_1}=\rho_{(\tau_1,E)}$ and $\nu_{\tau_2}=\rho_{(\tau_2,E)}.$  In the case of $\mathcal{I}_2$ for such congruences on $E$ we observe that $\tau_1\cap\tau_2=\iota$ where $\iota$ is the trivial congruence.  It is clear from \cref{lcons} that
\begin{align*}
(\tau_1,S)\Theta\cap(\tau_2,S)\Theta & = (\nu_{\tau_1}\vee\chi_S)\cap (\nu_{\tau_2}\vee\chi_S)\\
& \neq \chi_S \\
& = \nu_{\iota}\vee\chi_S\\
& = \nu_{\tau_1\cap\tau_2}\vee\chi_S\\
& = (\tau_1\cap\tau_2,S\cap S)\Theta
\end{align*}
thus $\Theta$ is not a join homomorphism.

\section{Two-sided Congruences}

We will now briefly discuss the lattice of two sided congruences, which we regard as a subset of the lattice of left congruences.  Given two congruences we recall that their join as congruences is equal to their join regarded as equivalence relations.  Hence the $\mathfrak{C}(S)$ is a sublattice of $\mathfrak{LC}(S).$  For a two sided congruence $\rho$ the kernel is an inverse subsemigroup, therefore 
$$\ker/(\rho)=\ker/(\rho\cap\mathcal{R})=\ink/(\rho).$$

It was first shown that two sided congruences on inverse semigroups are uniquely determined by the trace and kernel by Scheiblich \cite{Scheiblich}.  Pairs that arise as the trace and kernel of a congruence are commonly termed congruence pairs.  The following characterisation is due to Green \cite{green}.

\begin{deff}[Congruence Pair \cite{green}]
Let $T$ be a self conjugate full inverse subsemigroup of $S$ and let $\tau$ be a congruence on $E$ with $N(\tau)=S.$  Then say $(\tau,T)$ is a congruence pair if it satisfies:
\begin{enumerate}[label={(P\arabic*)}]
\item for $x\in S$ and $e\in E$ if $xe\in T$ and $e\ \tau\ x^{-1}x$ then $x\in T;$\label{cp1}
\item for each $x\in T$ we have $xx^{-1}\ \tau\ x^{-1}x.$\label{cp2}
\end{enumerate}
\end{deff}

\begin{theorem}[Kernel-trace description of two sided congruences on inverse semigroups; \cite{green}]\label{P1}
Let $(\tau,T)$ be a congruence pair for $S,$ and define:
$$P_{(\tau,T)}= \{(a,b)\ |\ a^{-1}a\ \tau\ b^{-1}b,\ ab^{-1}\in T \}.$$
Then $P_{(\tau,T)}$ is a congruence on $S.$  Moreover if $\rho$ is a congruence on $S$ then $(\tra/(\rho),\ker/(\rho))$ is a congruence pair for $S$ and $\rho=P_{(\tra/(\rho),\ker/(\rho))}.$
\end{theorem}

In the following we will show that this follows in straightforward fashion from the description of one-sided congruences in terms of the inverse kernel and trace.  It is clear that $\rho=\rho_{(\tau,T)}$ is a two sided congruence if and only if the left and right congruences this corresponding to $(\tau,T)$ are equal.

We use subscript $L,R$ to differentiate between left or right congruences, and recall that for an inverse congruence pair the corresponding left and right congruences are (resp.):
\begin{gather*}
\rho_L= \{ (a,b)\ |\ a^{-1}b\in T,\ a^{-1}bb^{-1}a\ \tau\ a^{-1}a,\ b^{-1}aa^{-1}b\ \tau\ b^{-1}b \}\\
\rho_R= \{ (a,b)\ |\ ab^{-1}\in T,\ ab^{-1}ba^{-1}\ \tau\ aa^{-1},\ ba^{-1}ab^{-1}\ \tau\ ab^{-1} \}.
\end{gather*}

\begin{prop}\label{prop2s1}
Let $(\tau,T)$ be an inverse congruence pair and let $\rho_L,\rho_R$ be the corresponding left and right congruences.  Then $(\tau,T)$ is a congruence pair if and only if 
$\rho_L=\rho_R.$
\end{prop}
\begin{proof}
Initially we assume that $\rho_l=\rho_r=\rho.$  Since $(\tau,T)$ is a inverse congruence pair we have that $T$ is a full inverse subsemigroup.  Also as $\rho$ is two sided we have that $\ink/(\rho)=T=\ker/(\rho).$

We first establish that $T$ is self conjugate.  Suppose that $b\in T$ so $b\ \rho\ bb^{-1}$.  As $\rho$ is a two sided congruence we obtain $aba^{-1}\ \rho\ abb^{-1}a^{-1}.$  As $abb^{-1}a^{-1}\in E,$ we have $aba^{-1}\in \ker/(\rho)=T,$ so $T$ is self conjugate. 

Next we show that $N(\tau)=S.$  Suppose that $e\ \tau\ f$ and $a\in S.$  As $\rho$ is a two sided congruence it is immediate that $aea^{-1} \rho\ afa^{-1},$ so $a^{-1}ea\ \tau\ a^{-1}fa,$ and thus $a\in N(\tau).$

We now establish \ref{cp1}.  Suppose that $ae\in T,$ and $e\ \tau\ a^{-1}a;$ we need that $a\in T.$  As $N(\tau)=S$ we have that $a\in N(\tau)$ so we conjugate $e\ \tau\ a^{-1}a$ by $a$ to obtain $aea^{-1}\ \tau\ aa^{-1},$ and we note that $ae=(aea^{-1})a$.  Then from \ref{lcpdeff2} (from \cref{lcpdeff}) it is immediate that $a\in T$. 

To establish \ref{cp2} we apply left and right versions of \cref{lem1} to get that 
$$\{a\ |\ a\ \rho\ a^{-1}a \}=T=\{a\ |\ a\ \rho\ aa^{-1} \}.$$
Then we get that if $a\in T$ then $a^{-1}a\ \rho\ a\ \rho\ aa^{-1}.$  

For the converse we suppose that $(\tau,T)$ is a congruence pair.  We first note that this implies that $(\tau,T)$ is an inverse congruence pair, as $T\subseteq S=N(\tau)$, and if \ref{cp1} holds then it is immediate that \ref{lcpdeff2} holds. 

We show that $\rho_L\subseteq \rho_R.$  Suppose $a\ \rho_L\ b,$ so $a^{-1}b\in T$ and $a^{-1}bb^{-1}a\ \tau\ a^{-1}a, b^{-1}aa^{-1}b\ \tau\ b^{-1}b.$  As $a^{-1}b\in T$ by \ref{cp2} we have that $b^{-1}aa^{-1}b\ \tau\ a^{-1}bb^{-1}a.$  Then we have:
$$a^{-1}a\ \tau\ a^{-1}bb^{-1}a\ \tau\ b^{-1}aa^{-1}b\ \tau\ b^{-1}b.$$
Since $N(\tau)=S$ we conjugate the relation $a^{-1}a\ \tau\ b^{-1}b$ by $a,b$ thus:
$$aa^{-1}\ \tau\ ab^{-1}ba^{-1},\ bb^{-1}\ \tau\ ba^{-1}ab^{-1}.$$
We also conjugate $a^{-1}a\ \tau\ a^{-1}bb^{-1}a$ by $a,$ and conjugate $b^{-1}b\ \tau\ b^{-1}aa^{-1}b$ by $b$ to obtain 
$$aa^{-1}\ \tau\ aa^{-1}bb^{-1}\ \tau\ bb^{-1}.$$
Since $b^{-1}a\in T,$ and $T$ is self conjugate, we have $ab^{-1}aa^{-1}\in T.$  Also
$$aa^{-1}\ \tau\ bb^{-1}\ \tau\ ba^{-1}ab^{-1}=(ab^{-1})^{-1}(ab^{-1}),$$
so \ref{cp1} with $x=ab^{-1}$ and $e=aa^{-1}$ gives that $ab^{-1}\in T.$  Whence we shown that $a\ \rho_R\ b.$

The dual argument gives that $\rho_R\subseteq \rho_L,$ hence the two are equal.
\end{proof}

To complete a proof of \cref{P1} it then suffices to show that when $\rho_L=\rho_R=\rho$ the two sided congruence reduces to the stated form.  We note that in the proof of \cref{prop2s1} we saw that when $(\tau,T)$ is a congruence pair and $a\ \rho_L\ b$ we have that $a^{-1}a\ \tau\ b^{-1}b.$  Since this is exactly when $\rho_L=\rho_R$ it is immediate that $\rho_L\subseteq \{(a,b)\ |\ a^{-1}a\ \tau\ b^{-1}b,\ ab^{-1}\in T\}=P_{(\tau,T)}.$

The other inclusion is also straightforward, suppose that $a^{-1}a\ \tau\ b^{-1}b$ and $ab^{-1}\in T.$  Then conjugating $a^{-1}a\ \tau\ b^{-1}b$ by $a,b$ gives $aa^{-1}\ \tau\ ab^{-1}ba^{-1},$ and $bb^{-1}\ \tau\ ba^{-1}ab^{-1}.$  Hence $P_{(\tau,T)}\subseteq \rho_R.$

\section{The Bicyclic Monoid}

Descriptions of one-sided congruences on the bicyclic monoid are known (\cite{nico} and \cite{duchamp1986etude}).  However it is an illuminating illustration of our techniques to apply the inverse kernel trace approach to the lattice of left congruences. We will use the following description of the bicyclic monoid: $B=\mathbb{N}^0\times \mathbb{N}^0$ with multiplication:
$$(a,b)(c,d)=(a-b+t,d-c+t)$$
where $t=\max\{b,c\}.$

Initially we describe the lattice of full inverse subsemigroups of $B,$ and the lattice of congruences on $E(B).$  For the former we appeal to the work of Jones \cite{jones} and Descal{\c{c}}o and Ru{\v{s}}kuc \cite{subsBicyclic}.  

\begin{deff}
For $k,d\in \mathbb{N}^0$ define:
$$T_{k,d}= \{(x,y)\ |\ x,y\geq k,\ d\! \mid\! x-y \}$$
\end{deff}

We note that each element in $T_{k,d}$ is of the form $(i,i+md)$ or $(i+md,i)$ for some $i\geq k$ and $m\geq 0.$  Together with $E(B),$ the $T_{k,d}$'s form a complete list of all full inverse subsemigroups of $B.$ 

\begin{theorem}[Theorem 7.1; \cite{subsBicyclic}]\label{byinv}
For $k\geq 0,\ d\geq 1,$ $T_{k,d}$ is a full inverse subsemigroup of $B$.  Moreover if $T\neq E(B)$ is a full inverse subsemigroup of $B,$ then $T=T_{k,d}$ for some $k\geq 0,\ d\geq 1.$ 
\end{theorem}

We note that $T_{k,d}\subseteq T_{j,c}$ if and only if $j\leq k,$ and $c\!\mid\! d.$  Let  $\mathbf{C}$ be the positive integers under the reverse of the usual order, and let $\mathbf{N}$ be the lattice consisting of the natural numbers with $m\leq_\mathbf{N} n$ if $n\mid m.$  Then $\mathfrak{V}(B)\cong (\mathbf{C}\times\mathbf{N})^0,$ where by $\mathbf{L}^0$ we mean the lattice $\mathbf{L}$ with a $0$ adjoined.  The $0$ of $\mathfrak{V}(B)$ corresponds to $E(B).$

We now consider the trace lattice.  Idempotents in $B$ are of the form $(x,x)$ for some $x\in \mathbb{N}^0,$ and $(x,x)(y,y)=(\max\{x,y\},\max\{x,y\}).$  Thus $E(B)$ is isomorphic to the lattice $\mathbf{C}.$  Congruences on a chain of idempotents can be viewed as partitions of the chain $\mathbf{C},$ hence a congruence on $E(B)$ corresponds to a partition of $\mathbb{N}^0.$  We note that a partition of $\mathbb{N}^0$ is determined by the set of the maximum (under the usual order on $\mathbb{N}^0$) element of each equivalence class (we note that in the case that there is an infinite equivalence class this corresponds to a finite set).  This gives a bijection between the set of congruences on $\mathbb{N}^0$ and $\mathbb{P}(\mathbb{N}^0)$ the powerset of $\mathbb{N}^0.$  We next observe that under this correspondence the ordering on the congruences becomes the reverse of the usual subset inclusion ordering on $\mathbb{P(N}^0).$  We write $\mathbf{P}$ for this lattice.

The next step is to compute the normaliser for each trace.  To do this it is helpful to establish the following notation.

\begin{deff}
Let $\tau$ be a congruence on $E(B).$  Let $\Gamma(\tau)=\{c_1,c_2,\dots\}$ be the set of integers $a,$ such that $(a,a)$ is maximum in its congruence class.  Also let $\Xi(\tau)=\{m_1,m_2,m_3\dots\}$ be the sequence of integers corresponding to sizes of the finite congruence classes.
\end{deff}

We note that for a congruence $\tau,$ $\Gamma(\tau)$ and $\Xi(\tau)$ are both finite if and only if $\tau$ has an infinite congruence class.  We also observe that $c_u=-1 + \sum_{i=1}^u m_i ,$ and $m_u=c_u-c_{u-1}.$  The 

\begin{deff}\label{evperdef}
Let $\tau$ be a congruence on $E(B)$ with no infinite congruence class.  Let $\Gamma(\tau)=\{m_1,m_2,\dots\},$ and $\Xi(\tau)=\{c_1,c_2,\dots\}$.  We say that $\tau$ is eventually periodic if there are $r,p\geq 1$ such that 
$$ s\geq r \implies m_s=m_{s+p}. $$
For $r,p$ chosen to be minimum such that this holds let $k=c_{r-1}+1$ if $r\geq 2$ or $k=0$ if $r=1,$ and let $d=\sum_{j=r}^{p-1} m_j.$  Then $d$ is the period of $\tau$ and we say that $\tau$ is $d$-periodic after $k.$
\end{deff}

We note that $r,p$ can certainly be chosen to both be minimum.  For, there is a shortest repeating pattern in $\Gamma(\tau),$ the length of which we set to $p,$ and there is then an earliest point this pattern starts, which we call $r.$  Let $\tau$ be an eventually periodic congruence on $E(B),$ and let $p,r$ be as in \cref{evperdef} chosen to be minimum.  Let $\Gamma(\tau)=\{m_1,m_2,\dots\},$ and $\Xi(\tau)=\{c_1,c_2,\dots\}.$  Then define $l(\tau)= c_r-\min\{m_{r-1},m_{r+p-1}\}+1.$

\begin{lemma}\label{perlem}
Let $\tau$ be a congruence on $E(B)$ which is $d$-periodic after $k,$ and let $l=l(\tau).$  Let $x,y\geq l,$ then $(x,x)\ \tau\ (y,y)$ if and only if $(x+d,x+d)\ \tau\ (y+d,y+d).$
\begin{proof}
Suppose that $\tau$ is $d$-periodic after $k$ with $r,p$ as before and let $\Gamma(\tau)=\{m_1,m_2,\dots\},$ and $\Xi(\tau)=\{c_1,c_2,\dots\}$.  We observe that as the sequence of $m_u$ repeats for $u\geq r$ for any $q\geq r$ we have $\sum_{i=q}^{q+p-1} m_i =d.$  Hence for any $u\geq r$ we have that $c_{u+p}=c_u +\sum_{i=u+1}^{u+p} m_i = c_u +d.$

Suppose  $(x,x)\ \tau\ (y,y).$   Initially suppose $k\leq x\leq y$ then there is some $u\geq r$ such that 
$$c_u < x \leq y \leq c_{u+1}.$$
Then we note that 
$$c_{u+p} = c_u + d < x+d\leq y+d\leq c_{u+1} + d  = c_{u+1+p}.$$
Thus $(x+d,x+d)\ \tau\ (y+d,y+d).$  

We note that since $(x,x)\ \tau\ (y,y)$ if $l\leq x<k$ then $l \leq y<k.$  If $l\leq x\leq y <k$ then we know that $c_{r+p-1}\leq x+d\leq y+d\leq c_{r+p}.$  Thus $(x+d,x+d)\ \tau\ (y+d,y+d).$

Conversely suppose that $(x+d,x+d)\ \tau\ (y+d,y+d).$  Then there is some $v$ such that $c_v < x+d\leq y+d\leq c_{v+1}.$  If $x,y\geq k$ we know that $r\ \leq v-p.$  Then $c_{v-p}=c_v-d$ and we observe
$$ c_{v-p}=c_v-d < (x+d)-d \leq (y+d)-d \leq c_{v+1}-d = c_{v-p+1}. $$
Hence $(x,x)\ \tau\ (y,y).$  If $l\leq x,y < k$ then as $c_{r-1}< l$ we have that $(x,x)\ \tau\ (y,y).$ 
\end{proof}
\end{lemma}

\begin{lemma}\label{perlem2}
Let $\tau$ be a congruence on $E(B).$  Suppose there are $k,d$ such that for $x,y\geq k$ we have $(x,x)\ \tau\ (y,y)$ if and only if $(x+d,x+d)\ \tau\ (y+d,y+d).$  Then either $\tau$ has an infinite congruence class, or $\tau$ is eventually periodic with period $d^\prime\mid d.$
\begin{proof}
Suppose $\tau$ has no infinite congruence class.  Let $\Gamma(\tau)=\{m_1,m_2,\dots\},$ and $\Xi(\tau)=\{c_1,c_2,\dots\}$.  Let $r$ be the least integer such that $k\leq c_r,$ and let $q=c_r.$  We claim that $q\leq k+d.$  Suppose not, then $(k,k)\ \tau\ (k+x,k+x)$ for $x=0,2,\dots,d+1.$  But then as $(k+1,k+1)\ \tau\ (k+2,k+2),$ we obtain $(k+d+1,k+d+1)\ \tau\ (k+d+2,k+d+2).$  Thus $(k,k)\ \tau\ (k+d+2,k+d+2).$  Inductively we get that $(k,k)\ \tau\ (k+x,k+x)$ for all $x\geq 0.$  This gives an infinite congruence class, which is a contradiction, so we certainly have that $q\leq k+d.$

We note that then $(q,q)\ \!\!\!\not\!\!\tau\ (q+1,q+1)$ and thus by the hypothesis we have that $(q+d,q+d)\ \! \not\!\tau\ (q+d+1,q+d+1).$  Therefore $q+d=c_{r+p}$ for some $p\geq 1.$  We then note that by a similar argument $c_{r+1}=c_{r+p+1}.$  Then 
$$m_{r+1}= c_{r+1}-c_r=(c_{r+1}+d)-(c_{r}+d) = c_{r+p+1}-c_{r+p}=m_{r+1+p}. $$
By an inductive argument it then follows that $m_s=m_{s+p}$ for all $s\geq r,$ thus $\tau$ is eventually periodic.  We also note that $d=\sum_{i=1}^{p-1} m_i.$  
Suppose $r^\prime\leq r$ and $p^\prime\leq p$ are chosen minimum such that $m_s=m_{s+p^\prime}$ for all $s\geq r^\prime.$  Then certainly $p^\prime\mid p$ and with $d^\prime=\sum_{i=1}^{p^\prime-1} m_i$ we obtain $d^\prime\mid d.$  We also note that $c_r\leq k+d,$ so certainly $k^\prime=c_{r^\prime}\leq k+d.$
\end{proof}
\end{lemma}

\begin{prop}
Let $\tau$ be a congruence on $E(B)$ such $\tau$ has an infinite congruence class with largest idempotent $(n,n).$  Then $N(\tau)= T_{n,1}.$
\begin{proof}
We first note that if $x,y\geq n$ then $(x,y)(s,s)(y,x)=(z_s,z_s),$ where $z_s=\max\{x-y+s,x\}\geq x.$  Hence if $(s,s)\ \tau\ (t,t)$ we have $z_s,z_t\geq n,$ thus $(z_s,z_s)\ \tau\ (z_t,z_t)$ thus $(x,y)\in N(\tau)$ thus $T_{n,1}\subseteq N(\tau).$

Suppose that $(x,y)\in B,$ with $x< n\leq y.$  Then $(y,y)\ \tau\ (n+y,n+y).$  However we observe that $(x,y)(y,y)(y,x)=(x,x),$ and $(x,y)(n+y,n+y)(y,x)=(x+n,x+n).$  Since $x<n$ we have that $(x,x)\ \not\!\tau\ (x+n,x+n),$ thus $(x,y)\notin N(\tau).$

Suppose finally that $x<y<n.$  We note that $(n,n)\ \tau\ (n+y,n+y).$  Then $(x,y)(n,n)(y,x)=(x-y+n,x-y+n),$ and $(x,y)(n+y,n+y)(y,x)=(n+x,n+x).$  Then as $x-y+n< n \leq n+x,$ we have that $(x-y+n,x-y+n)\ \not\!\tau\ (x+n,x+n).$  Hence $(x,y)\notin N(\tau).$  Thus $N(\tau)=T_{n,1}.$
\end{proof}
\end{prop}

We next compute the normaliser for an eventually periodic trace.

\begin{lemma}
Let $\tau$ be $d$-periodic after $k,$ and let $l=l(\tau).$  Then $N(\tau)= T_{l,d}.$
\begin{proof}
Suppose $(a,a+b)\in N(\tau) \cap B\backslash T_{k,d}.$  We observe that for $s\geq a,$ we have $(a+b,a)(s,s)(a,a+b)=(s+b,s+b),$ and for $s\geq a+b$ we have $(a,a+b)(s,s)(a+b,a)=(s-b,s-b).$  But then for $s,t\geq a$ we have that $$(s,s)\ \tau\ (t,t) \iff (s+b,s+b)\ \tau\ (t+b,t+b).$$
Then applying \cref{perlem2} we obtain that $\tau$ is $d^\prime$-periodic after $k^\prime$ with $d^\prime\mid b.$  As we know that $\tau$ has period $d$ we have that $d\mid b$ so if $a\geq k$ then $(a,a+b)\in T_{k,d},$ and we have a contradiction.

Thus we may assume that $a\leq k.$  Since $(a,a+b)\in N(\tau)$ and $N(\tau)$ is a full inverse subsemigroup containing $T_{k,d}$ and $d\mid b$ we have that $(a,a+d)\in N(\tau).$  Suppose that $x\geq a$ and  $(x,x)\ \tau\ (x+1,x+1).$  Then
$$(x+d,x+d)=(a+d,a)(x,x)(a,a+d)\ \tau\ (a+d,a)(x+1,x+1)(a,a+d) = (x+d+1,x+d+1).$$
Then we note that $x+d\geq a+d,$ so
$$(x,x)= (a,a+d)(x+d,x+d)(a+d,a)\ \tau\ (a,a+d)(x+d+1,x+d+1)(a+d,a) = (x+1,x+1). $$
Thus for $x\geq a$ we have $(x,x)\ \tau\ (x+1,x+1)$ if and only if $(x+d,x+d)\ \tau\ (x+1+d,x+1+d).$

Suppose that $a<l,$ then as $N(\tau)$ is full we have $(l-1,l-1+d)\in N(\tau).$  However from the definition of $l$ exactly one of $(l-1,l-1)\ \!\!\not\!\tau\ (l,l)$ or $(l-1+d,l-1+d)\ \!\!\not\!\tau\ (l+d,l+d).$  But with $l-1=x$ this gives a contradiction.  Hence we must have $a\geq l,$ and thus $(a,a+b)\in T_{l.d}$

We now show that $T_{l,d}\subseteq N_{\tau}.$  Take $(a,a+bd)\in T_{l,d}$ and suppose that $(s,s)\ \tau\ (t,t),$ then we consider
\begin{align*}
(a,a+bd)(s,s)(a+bd,a) &= (x-bd,x-bd)  & \text{   for } x=\max\{a+bd,s\} \\
(a,a+bd)(t,t)(a+bd,a) &= (y-bd,y-bd)  & \text{   for } y=\max\{a+bd,t\}
\end{align*}
We need that $(x-bd,x-bd)\ \tau\ (y-bd,y-bd).$  If $s \leq t\leq a+bd$ then this is immediate as $x-bd=a=y-bd.$  Suppose $a+bd\leq s\leq t,$ then in particular $l \leq a \leq s-bd\leq t-bd.$  Then by repeated application of \cref{perlem} we have that $(s-bd,s-bd)\ \tau\ (t-bd,t-bd)$ if and only if $(s,s)\ \tau\ (t,t).$  Suppose finally that $s\leq a+bd\leq t.$  Then $ (a+bd,a+bd)\ \tau\ (t,t)$ so we apply the previous argument to get that $(a,a)\ \tau\ (t-bd,t-bd).$  

Similarly we can show that if $(s,s)\ \tau\ (t,t)$ then $(a+bd,a)(s,s)(a,a+bd)\ \tau\ (a+bd,a)(t,t)(a,a+bd),$ and thus we obtain $(a,a+bd)\in N(\tau).$  It is then clear that $T_{l,d}\subseteq N(\tau).$
\end{proof}
\end{lemma}

\begin{prop}
Let $\tau$ be a congruence on $E=E(B),$ with no infinite congruence class. Then $\tau$ is eventually periodic if and only if $N(\tau)\neq E(B).$ 
\begin{proof}
We first note that if $\tau$ is $d$-periodic after $k$ then we have $T_{k,d}\subseteq N(\tau).$  Thus $N(\tau)\neq E(B).$

Suppose now that $N(\tau)\neq E(B).$  Choose $a\geq0, b\geq 1$ with $(a,a+b)\in N(\tau).$  If $v\geq a+b$ then $(a+b,a)(v,v)(a,a+b)=(v+b,v+b),$ and $(a,a+b)(v,v)(a+b,a)=(v-b,v-b).$  Then we note that for $s,t\geq a+b$ we have that $$ (s,s)\ \tau\ (t,t) \iff (s+b,s+b)\ \tau\ (t+b,t+b). $$
Thus by \cref{perlem2} since $\tau$ has no infinite congruence class we have that $\tau$ is eventually periodic.
\end{proof}
\end{prop}

We have then described the normaliser of every congruence on $E(B).$

\begin{lemma}
Let $\tau$ be a congruence on $B$ and let $\nu_L,\nu_R$ be minimum left and right congruences on $B$ with this trace.  Then
\begin{gather*}
(a,b)\ \nu_L\ (c,d) \iff (b,b)\ \tau\ (d,d), \text{ and } b-a=d-c,\\
(a,b)\ \nu_R\ (c,d) \iff (a,a)\ \tau\ (c,c), \text{ and } b-a=d-c.
\end{gather*}
\begin{proof}
This is a straightforward application of the description of minimum one-sided congruences given in \cref{M1}.
\end{proof}
\end{lemma}

Given a trace, we now want to describe which full inverse subsemigroups contained in the normaliser are saturated by $\nu.$  Given a congruence $\tau$ which contains an infinite congruence class containing maximum idempotent $(n,n)$ we recall that $N(\tau)=T_{n,1}.$  We then note that on $T_{n,1}$ we have that $(a,b)\ \nu\ (c,d)$ if $b-a=d-c.$  Thus the subsemigroups saturated by $\nu$ are precisely $T_{n,d}$ for $d\geq 1.$  We also note that as a lattice this is isomorphic to $\mathbf{N}.$

If $\tau$ has no infinite congruence class and is not eventually periodic, then the normaliser is $E(B),$ so the only inverse congruence pair containing $\tau$ is $(\tau,E(B)).$

Suppose $\tau$ is eventually periodic, with $\Xi(\tau)=\{c_1,c_2,\dots\}$ and $r,p$ defined as usual.  Also suppose $\tau$ has period $d$ and let $l=l(\tau).$  Then $N(\tau)=T_{l,d}.$ We then observe that $T_{k,c}\subseteq T_{l,d}$ if $k\leq l$ and $c\mid d.$  It is also straightforward that $T_{k,c}$ is saturated by $\nu$ only if $k=l$ or $k=c_u+1$ for some $u\geq r.$  We note that as a lattice this is isomorphic to $(\mathbf{N}\times \mathbf{C})^0\cong \mathfrak{V}(B).$ 

\begin{theorem}
Let $\tau$ be a congruence on $E(B)$ with $\Xi(\tau)=\{c_1,c_2,\dots\}.$  Then $(\tau,T)$ is an inverse congruence pair for $B$ if and only if at least one of the following holds:
\begin{enumerate}[label={(\roman*)}]
\item $T=E(B);$
\item $\tau$ has an infinite congruence class $\{(x,x)\ |\ x\geq n\}$ and there is $c\geq 1$ with $T=T_{n,c};$
\item $\tau$ is $d$-periodic after $k$ and $T=T_{j,c}$ with $d\mid c$ and either $j=l(\tau)$ or $j=c_u+1$ for some $u$ with $c_u\geq k-1.$
\end{enumerate}
\end{theorem}

\section{Finitely generated congruences}

Given a set $Z\subseteq S$ write $|Z|$ for the inverse subsemigroup generated by $Z.$  Given a generating set $Z$ for an inverse subsemigroup we will assume that $Z=Z^{-1},$ i.e. $Z$ contains all inverses of elements in $Z.$  Given a generating set for a left congruence $R\subseteq S\times S$ we will assume that $R$ is symmetric, i.e. if $(a,b)\in R$ then $(b,a)\in R.$ 

Several properties of semigroups are related to whether one-sided congruences are finitely generated.  We will see that for inverse semigroups finite generation of left congruences is closely tied to finite generation of the trace and the inverse kernel.  Initially we have a technical lemma regarding generating sets for one-sided congruences.

\begin{lemma}\label{gen set}
Let $H\subseteq S\times S,$ be a symmetric set, and let $\rho=\langle H\rangle.$  Then there exists a symmetric set $H^\prime$ such that $\rho=\langle H^\prime\rangle$ and
$$H^\prime\subseteq (E(S)\times E(S))\cup \{ (e,a)\ |\ a\ \mathcal{R}\ e\} \cup \{ (a,e)\ |\ a\ \mathcal{R}\ e\}.$$
Moreover, if $H$ is finite then $H^\prime$ is also finite.
\end{lemma}
\begin{proof}
Suppose that $a\ \rho\ b.$  By \cref{biggy} we have that $a^{-1}b\in \ink/(\rho),$ hence $a^{-1}b\ \rho\ a^{-1}bb^{-1}a.$  Also from $a\ \rho\ b$ we have that $a^{-1}a\ \rho\ a^{-1}bb^{-1}a$ and $b^{-1}aa^{-1}b\ \rho\ b^{-1}b.$

Conversely we note that if the three relations: $a^{-1}b\ \rho\ a^{-1}bb^{-1}a,\ a^{-1}a\ \rho\ a^{-1}bb^{-1}a,$ and $b^{-1}aa^{-1}b\ \rho\ b^{-1}b$ hold then we also have $a\ \rho\ b.$  Indeed, from $a^{-1}b\ \rho\ a^{-1}bb^{-1}a$ we have $b^{-1}a(a^{-1}b)\ \rho\ b^{-1}a(a^{-1}bb^{-1}a)=b^{-1}a.$  We then observe
$$a=aa^{-1}a\ \rho\  aa^{-1}bb^{-1}a = bb^{-1}a\ \rho\ bb^{-1}aa^{-1}b\ \rho\ bb^{-1}b = b.$$
Hence we may replace each $(a,b)\in H$ with $3$ pairs in the set claimed.
\end{proof}

\begin{corr}
Every finitely generated left congruence $\rho$ on $S$ can be written as the join $\chi\vee \nu_\tau$, where $\tau$ is a finitely generated congruence on $E(S)$ and $\chi$ is a finitely generated idempotent separating left congruence on $S.$
\end{corr}

\begin{deff}
A full inverse subsemigroup $T\subseteq S$ is said to be {\em almost finitely generated} if there exists a finite set $X$ such that $T=| X\cup E(S) |.$
\end{deff}

This notion exactly captures which full inverse subsemigroups are inverse kernels of finitely generated idempotent separating congruences.

\begin{lemma}
Let $\chi$ be an idempotent separating left congruence on $S$, then $T=\ink/(\chi)$ is almost finitely generated if and only if $\chi$ is finitely generated.
\begin{proof}
Suppose that $T$ is almost finitely generated.  Suppose $T=| X\cup E|,$ where $X$ is a finite set.  Let $Y= \{(x,xx^{-1})\ |\ x\in X\};$ we claim that $\chi= \langle Y \rangle.$  Certainly $X\subseteq \ink/(\langle Y\rangle),$ and as the inverse kernel is a full inverse subsemigroup we have that $T\subseteq \ink/(\langle Y\rangle).$  Since $\langle Y\rangle$ is an idempotent separating left congruence and is thus determined by its inverse kernel we have that $\chi\subseteq \langle Y\rangle.$  Conversely we note that we certainly have $(x,xx^{-1})\in \chi$ for each $x\in X.$  Thus $Y\subseteq \chi,$ thus $\langle Y\rangle\subseteq \chi.$  Hence the two are equal, and thus $\chi$ is finitely generated.

For the converse we suppose that $\chi$ is finitely generated.  By \cref{gen set} we can choose a finite generating set $Q$ for $\chi$ such that $Q=\{(p,pp^{-1})\ |\ p\in P\}$ for some finite set $P\subseteq S.$  Then we claim that $\ink/(\chi)=| P|.$  It is immediate that $P\subseteq \ink/(\chi).$  Let $\zeta$ be the idempotent separating left congruence with inverse kernel equal to $| P|.$  It suffices to show that $\chi\subseteq \zeta.$  We then note that $(p,pp^{-1})\in \zeta$ for each $p\in P.$  Thus $Q\subseteq \zeta,$ and hence $\langle Q\rangle =\chi\subseteq \zeta.$  Thus we have that $\ink/(\chi)$ is finitely generated.
\end{proof}
\end{lemma}

We shall want to consider idempotent separating left congruences corresponding to almost finitely generated full inverse subsemigroups; for ease of notation if $Y$ is an inverse subsemigroup then we write $\chi_Y$ for the idempotent separating left congruence with inverse kernel $|Y\cup E(S)|.$  We note that if $Y$ is an inverse subsemigroup of $S$ then $|Y\cup E(S)|=YE(S)\cup E(S).$

We note that $\mathfrak{C}_{FG}(E),$ the lattice of finitely generated congruences on $E(S)$ is a sublattice of $\mathfrak{C}(E).$  Also $\mathfrak{V}_{AFG}(S),$ the lattice of almost finitely generated full inverse subsemigroup of $S$ is a join-subsemilattice of $\mathfrak{V}(S).$ 

\begin{corr}
Let $\rho$ be a finitely generated left congruence on $S.$  Then there are $T\in \mathfrak{V}_{AFG}(S)$ and $\tau\in \mathfrak{C}_{FG}(E)$ such that
$$\rho=\chi_T\vee\nu_\tau.$$
\end{corr}

We recall from \cref{embed thm} the function
$$ \Theta: \mathfrak{C}(E)\times\mathfrak{V}(S)\rightarrow \mathfrak{LC}(S);\ (\tau,T)\mapsto \nu_\tau\vee\chi_T. $$
The previous corollary gives that if $\rho$ is finitely generated then it is the image under $\Theta$ of a pair $(\tau,T)\in \mathfrak{C}_{FG}(E)\times \mathfrak{V}_{AFG}(S).$  In fact $\mathfrak{LC}_{FG}(S),$ the set of finitely generated left congruences on $S$ is precisely the image $(\mathfrak{C}_{FG}(E)\times \mathfrak{V}_{AFG}(S))\Theta.$ 

It is of interest to consider when every left congruence is finitely generated.  For inverse semigroups a it is possible to describe exactly these semigroups \cite{kozhukhov1980semigroups}.  A partial order $P$ is said to have the ascending chain condition if every increasing sequence is eventually constant.

\begin{deff}
A semigroup $S$ is {\em left Noetherian} if every left congruence on $S$ is finitely generated.  Equivalently the lattice of left congruences has the ascending chain condition.
\end{deff}

The preceding remarks show that $S$ is left Noetherian if and only if every left congruence is join of a finitely generated trace minimal left congruence and a finitely generated idempotent separating left congruence.  The following is a straightforward observation about the ascending chain condition on partial orders.

\begin{lemma}\label{acclem}
Let $P,Q$ be partial orders that have the ascending chain condition, and let $R\subseteq P$ be a suborder.  Then $R, P\times Q$ have the ascending chain condition.
\end{lemma}

\begin{theorem}
Let $S$ be an inverse semigroup.  The lattice $\mathfrak{LC}(S)$ has the ascending chain condition if and only if $\mathfrak{V}(S)$ and $\mathfrak{C}(E)$ have the ascending chain condition.
\begin{proof}
Suppose that $\mathfrak{V}(S),\mathfrak{C}(E)$ have the ascending chain condition.  By \cref{acclem} we have that $\mathfrak{V}(S)\times\mathfrak{C}(E)$ has the ascending chain condition.  However we know that as partial orders we have that $\mathfrak{LC}(S)\subseteq \mathfrak{V}(S)\times\mathfrak{C}(E).$  Thus from \cref{acclem} we obtain that $\mathfrak{LC}(S)$ has the ascending chain condition.

Suppose conversely that $\mathfrak{LC}(S)$ has the ascending chain condition.  We note that we have lattice embeddings $\mathfrak{V}(S)\hookrightarrow\mathfrak{LC}(S)$ and $\mathfrak{C}(E)\hookrightarrow\mathfrak{LC}(S).$  Then certainly we have $\mathfrak{V}(S),\mathfrak{C}(E)\subseteq \mathfrak{LC}(S)$ as partial orders.  Thus by \cref{acclem} we have that $\mathfrak{V}(S),\mathfrak{C}(E)$ have the ascending chain condition.
\end{proof}
\end{theorem}

The following are standard results concerning the ascending chain condition on the two lattices: full inverse subsemigroups and congruences on $E.$

\begin{result}
Let $E$ be a semilattice.  Then $\mathfrak{C}(E)$ has the ascending chain condition if and only if $E$ is finite.
\end{result}

\begin{result}
The lattice $\mathfrak{V}(S)$ has the ascending chain condition if and only if every full inverse subsemigroup is finitely generated
\end{result}

Noting that a semilattice is finite if and only if it is finitely generated the usual formulation for the classification of left Noetherian inverse semigroups is immediate.

\begin{theorem}[Theorem 4.3, \cite{kozhukhov1980semigroups}]
Let $S$ be an inverse semigroup.  Then $S$ is left Noetherian if and only if every full inverse subsemigroup is finitely generated.
\end{theorem}

\section{Generating sets for left congruences}

We now turn our attention to a more detailed discussion of what the inverse kernel approach to left congruences tells us about the structure of a left congruence given a generating set.  

\begin{lemma}\label{tracejoin}
Let $\tau$ be a congruence on $E,$ $Y\subseteq S$ an inverse subsemigroup, and $\rho=\chi_Y\vee\nu_\tau.$  Then 
$$\tra/(\rho)=\langle \tau\cup\{(aea^{-1},afa^{-1}) \ |\ (e,f)\in\tau,\ a\in Y\}\rangle.$$
\end{lemma}
\begin{proof}
\cref{latprop2} gives that $\tra/(\rho)=\xi$ where $\xi$ is the least congruence on $E$ such that $\tau\subseteq \xi,$ and $\ink/(\chi_Y)\vee\ink/(\nu_tau)\subseteq N(\xi).$  We note that since the map $\tau\mapsto \nu_\tau$ is order preserving we have that if $\tau\subseteq \xi$ then $\ink/(\nu_\tau)\subseteq\ink/(\nu_\xi)\subseteq N(\xi),$ and also if $Y\subseteq V$ for a full inverse subsemigroup $V$ then $\ink/(\chi_Y)\subseteq V.$  Hence we may realise $\xi$ as the least congruence on $E$ such that $\tau\subseteq\xi$ and $Y\subseteq N(\xi).$

Write $X=\tau\cup\{(aea^{-1},afa^{-1}) \ |\ (e,f)\in\tau,\ a\in Y\}.$  Then it is clear that $X\subseteq \xi.$  We also note that $e\ \langle X\rangle\ f$ with sequences $(p_i,q_i)\in X,\ h_i\in E$ such that
$$e=h_1p_1,\ h_1q_1=h_2p_2,\dots\,\ h_nq_n=f$$
then we can observe that
\begin{gather*}
aea^{-1}=ah_1p_1a^{-1}=(ah_1a^{-1})(ap_1a^{-1}),\\
(ah_1a^{-1})(aq_1a^{-1})=ah_1q_1a^{-1}=ah_2p_2a^{-1}=(ah_2a^{-1})(ap_2a^{-1}),\ \dots \\
(ah_na^{-1})(aq_na^{-1})=ah_nq_na^{-1}=afa^{-1}.
\end{gather*}
We note that if $(p_i,q_i)\in X$ then for $a\in Y$ we have $(ap_ia^{-1},aq_ia^{-1})\in X,$ thus $aea^{-1}\ \langle X\rangle\ afa^{-1}.$  Hence we have $\tau\subseteq \langle X\rangle,$ and $Y\subseteq N(\langle X\rangle),$ thus $\xi=\langle X\rangle.$
\end{proof}

Given a left congruence $\rho$ we will now be interested in which $a\in S$ have $[a]_\rho=\{ a\}.$ 

\begin{lemma}
Let $\rho$ be a left congruence on $S,$ and $a\in S.$  If $[a^{-1}a]_\rho=\{a^{-1}a\},$ then $[a]_\rho=\{a\}.$
\begin{proof}
Suppose that $[a]_\rho\neq \{a\},$ so there is some $b\neq a$ such that $a\ \rho\ b.$  Then certainly $a^{-1}a\ \rho\ a^{-1}b,$ if $a^{-1}b\neq a^{-1}a,$ then there is nothing to show, so we assume that $a^{-1}b=a^{-1}a.$  This is equivalent to $a^{-1}a=b^{-1}a,$ and to $a\leq b.$  We note that by \cref{biggy} we have from $a\ \rho\ b$ that $b^{-1}b\ \rho\ b^{-1}aa^{-1}b.$  However we then have $b^{-1}b\ \rho\ b^{-1}aa^{-1}b=a^{-1}a.$  If $b^{-1}b=a^{-1}a$ then $a\ \mathcal{L}\ b,$ and thus we get $a=b,$ a contradiction.  Thus we have that $[a^{-1}a]\neq \{a^{-1}a\}.$
\end{proof}
\end{lemma}

\begin{lemma}\label{trivjoin}
Let $\rho_1,\rho_2$ be left congruences on $S.$  Let $\rho=\rho_1\vee\rho_2$ and $a\in S.$  Then $[a]_\rho=\{a\}$ if and only if $[a]_{\rho_1}=\{a\}$ and $[a]_{\rho_2}=\{a\}.$
\begin{proof}
This is immediate noting that the join is equal to the transitive closure of the two left congruences.
\end{proof}
\end{lemma}

Thus for a left congruence $\rho_{(\tau,T)}$ to determine which $a\in S$ have $[a]_\rho\neq \{a\}$ it suffices to calculate which $a$ have $[a^{-1}a]_{\nu_\tau}\neq \{a^{-1}a\},$ or $[a^{-1}a]_{\chi_T}\neq \{a^{-1}a\}.$  We then note that for an given a congruence $\tau$ on $E(S),$ the equivalence class $[e]_{\nu_\tau} \neq \{e\}$ if and only if $[e]_{\tau} \neq \{e\}.$  Also given an idempotent separating left congruence $\chi_T$ on $S,$ we note that $[e]_{\chi_T} \neq \{e\}$ if and only if $[e]_\mathcal{R}\cap T \neq \{e\}.$

It is of interest to consider semigroups for which the universal relation $\omega,$ is finitely generated as a left congruence.  For monoids this is equivalent to the monoid being of type left $\text{FP}_1$.  This has been studied for several classes of semigroups, and a classification for inverse semigroups is available \cite{tomandvicky}.  We now use our methodology to provide a direct proof of this result.

If $\omega$ is finitely generated then there is a finitely generated congruence on $E(S)$ and a finitely generated full inverse subsemigroup $Y$ such that $\omega=\nu_\tau\vee\chi_Y.$  We note that in this case we may assume that there is a finite set $X\subseteq E(S)$ such that $\tau=\langle X\times X\rangle.$  We can then apply \cref{tracejoin} to obtain that $\omega_E,$ the universal congruence on $E,$ is generated by $Z\times Z,$ where $Z= \tau\cup \{ (aea^{-1},afa^{-1})\ |\ a\in Y,\ e,f\in X\}.$ 

\begin{deff}
An idempotent $f$ is {\em maximal} in $S$ if for $e\in E(S)$ we have $f\leq e$ implies $e=f.$ 
\end{deff}

\begin{lemma}\label{gensetlem}
Let $Z\subseteq E(S)$ and let $X\subseteq Z\times Z,$ let $\tau=\langle X\rangle$ be a congruence on $E(S).$ Let $Y$ be an inverse subsemigroup of $S.$  Let $\rho=\nu_\tau\vee\chi_Y.$  Then 
\begin{enumerate}[label={(\roman*)}]
\item if $[e]_\rho\neq\{ e\}$ then $e\leq f$ for some $f\in E(Y) \cup Z.$
\end{enumerate}
Suppose now that $Z$ is finite and $Y$ is finitely generated, then
\begin{enumerate}[resume,label={(\roman*)}]
\item $|Y\cup Z|$ contains finitely many maximal idempotents,
\item if $[e]_\rho\neq\{e\}$ then $e\leq $ for some idempotent $m$ maximal in $E(|Y\cup X|),$ 
\item $[m]_\rho\neq\{ m\}$ for finitely many maximal idempotents $m.$ 
\end{enumerate}
\begin{proof}
We observe that by \cref{trivjoin} $[e]_\rho\neq \{e\}$ if and only if at least one of $[e]_{\chi_Y}\neq \{e\}$ or $[e]_{\nu_\tau}\neq \{e\}.$  We also note that 
\begin{align*}
[e]_{\chi_Y}\neq \{e\} & \implies \ink/(\chi_Y)\cap [e]_\mathcal{R}\neq \{e\}\\
& \implies (YE(S)\cup E(S))\cap [e]_\mathcal{R}\neq \{e\}\\
& \implies e\in YE(S)\\
& \implies e\leq f\in Y
\end{align*}
Also
\begin{align*}
[e]_{\nu_\tau}\neq \{e\} & \implies [e]_\tau \neq \{e\}\\
& \implies e\leq f\in Z
\end{align*}
This completes the proof of the first part.

We now assume that $Z$ is finite and $Y$ is finitely generated, say $Y=|W|$ with $W\subseteq S$ finite and inverse.  

Let $U^\prime=X\cup \{yy^{-1}\ |\ y\in W\}.$  Let $U$ be the set of elements in $U^\prime$ that are maximal in $U^\prime.$  So
$$U=\{m\in U^\prime\ |\ h\in U^\prime \text{ and } m\leq h \implies m=h\}.$$
We claim that $U$ is the set of maximal elements in $|W\cup Z|.$   We note that for $u\in U^\prime$ there is a $m\in U$ such that $u\leq m,$ and for $m_1\neq m_2\in U$ we have $m_1\not\leq m_2$ and $m_2\not\leq m_1.$  We then note that if $e\in E(|W\cup Z|),$ then $e=q_1q_2\dots q_k$ for some $q_i\in W\cup Z.$  We then observe that $e\leq q_1q_1^{-1}.$  Also $q_1q_1^{-1}\leq m,$ for some $m\in U.$  Suppose there is $m^\prime\in U$ with $m^\prime\leq e.$  Then $m^\prime\leq e\leq m,$ hence $m^\prime=e=m,$ and thus the elements in $U$ are maximal.  Since $U$ is finite we have also shown that $|Y\cup Z|$ contains finitely many maximal idempotents.  Also by $(i)$ we have that if $[e]_\rho\neq\{e\}$ then $e\leq f$ for some $f\in E(Y)\cup Z.$  We have that every $f\in E(Y)\cup Z$ has $f\leq m$ for some $m\in U,$ hence we have the third claim. 

The final part then follows immediately.
\end{proof}
\end{lemma}

\begin{lemma}\label{congcl}
Let $E$ be a semilattice such that every element is below a maximal idempotent, and $F\subset E$ be a subsemilattice containing all maximal idempotents.  Let $K=\{ e\in E\ |\ \exists f\in F \text{ with } f\leq e \}.$  Then $K$ is a congruence class of $\tau=\langle F\times F\rangle.$  
\end{lemma}
\begin{proof}
We first note that certainly for $k_1,k_2\in K$ we have $k_1\ \tau\ k_2,$ and also that $K$ is also a subsemilattice of $E.$  Suppose that $g\ \tau\ h$ with $h\in K.$  Then there exist $p_i,q_i\in F,$ $y\in E$ such that 
$$g=y_1p_1,\ y_iq_i=y_{i+1}p_{i+1},\ y_nq_n=h.$$
We note that $h\leq y_n,$ hence $y_n\in K.$  the same argument then gives that $y_i\in K$ for each $i$ and thus $g\in K.$   
\end{proof}

\begin{theorem}
Let $S$ be an inverse semigroup, then $\omega$ is finitely generated as a left congruence on $S$ if and only if $E(S)$ contains finitely many maximal elements and every idempotent is below a maximal idempotent, and there is a finitely generated subsemigroup $W\subseteq S$ such that for each $a \in S$ there is some $f\in E(W)$ with $af\in W.$
\end{theorem}
\begin{proof}
Suppose first that $\omega$ is generated by a finite set, with $Y\subseteq S,\ X\subseteq E$ finite sets and $\tau=\langle X\times X\rangle,$ and $\omega=\chi_{|Y|}\vee\nu_\tau.$  

Applying \cref{gensetlem} we have that $S$ has finitely many maximal idempotents, and that every idempotent is below some maximal idempotent.  We may assume that all maximal idempotents of $S$ lie in $X.$

Let $Z= X\cup\{ aea^{-1}\ |\ a\in |Y|,\ e,\in X\},$ then by \cref{tracejoin} we have that $\omega_E=\langle Z\times Z\rangle.$  Let $F$ be the subsemilattice generated by $Z,$ and note that $F$ contains all maximal idempotents of $S$.  Let $K=\{ e\in E\ |\ \exists f\in F \text{ with } f\leq e\}.$  Then by \cref{congcl} $K$ is a congruence class of $\omega_E,$ so $K=E,$ hence given $e\in E$ there is some $f\in Z$ such that $f\leq e.$  

Recall that $|Y\cup E|=E \cup |Y|E,$ and note that as $Y,X$ are finite $W=|X\cup Y|$ is finitely generated.  We also observe that $Z\subseteq W.$  Since $\chi_{|Y|}\vee\nu_\tau=\omega$ and $\ink/(\omega)=S$ we have that the $\nu_{\omega}$ closure of $|Y\cup E|$ is $S,$ i.e. $S=\bigcup_{t\in |Y\cup E|} [t]_{\nu_\omega}.$  Hence for each $a\in S$ there is some $e\in E$ with $ae\in|Y\cup E|.$ 

If $ae\in (|Y\cup E|)\cap E$ then we know there is some $f\in Z\subseteq W$ such that $f\leq ae.$  Then as $fae=f=ea^{-1}f$ we have
$$ af= aea^{-1}f=(fae)(ea^{-1}f)=ff=f, $$
so $af\in W.$

For $ae\in |Y\cup E|\backslash E$ there is some $b\in |Y|,$ and $e^\prime\in E(S)$ with $be^\prime=ae.$  Choose $f\in Z$ with $f\leq ee^\prime.$  Then $af=aef=be^\prime f=bf\in W.$

Conversely assume that $S$ is as claimed, and $W=|V|$ for some finite inverse set $V\subseteq S.$  Let $T=|V\cup E|,$ and $X=\{ vv^{-1}\ |\ v\in V\}.$  We may also assume that all maximal idempotents are in $X.$ let $\tau=\langle X\times X\rangle.$  Let $\rho=\chi_T\vee\nu_\tau.$  We first note that for $a,b\in W$ we have $a\ \rho\ b,$ indeed suppose $a=v_1\dots v_n,$ then:
$$v_1v_1^{-1}\ \chi\ v_1v_1^{-1}v_1\ \nu\ v_1v_2v_2^{-1}\ \chi\ v_1v_2v_2^{-1}v_2\ \dots\ \nu\ v_1\dots v_nv_n^{-1}\ \chi\ v_1\dots v_n=a.$$
Similarly we obtain $v_1v_1^{-1}\ \rho\ b,$ and hence $a\ \rho\ b.$  In particular this gives that for $e,f\in E(W)$ we have $e\ \rho\ f.$ 

By assumption for $a\in S$ we have some $f\in E(W)$ with $af\in W.$  Choose $m$ a maximal idempotent with $m\geq a^{-1}a.$  As $m,f\in E(W)$ we then have $m\ \rho\ f,$ and thus $a=am\ \rho\ af.$  Hence $\rho=\omega.$ 
\end{proof}

\section*{Acknowledgements}
This project forms part of the work toward my PhD at the University of York, funded by EPSRC.  I would like to thank my supervisor, Professor Victoria Gould, for all her help and guidance during this project.

\begin{small}
\nocite{*}
\bibliographystyle{abbrv}
\bibliography{references}

\begin{thebibliography}{10}

\bibitem{subsBicyclic}
L.~Descal{\c{c}}o and N.~Ru{\v{s}}kuc.
\newblock Subsemigroups of the bicyclic monoid.
\newblock {\em International J. Algebra and Computation}, 15(01):37--57, 2005.

\bibitem{duchamp1986etude}
G.~H.~E. Duchamp.
\newblock {\'E}tude du treillis des congruences {\`a} droite sur le
  mono{\"\i}de bicyclique.
\newblock {\em Semigroup Forum}, 33(1):31--46, 1986.

\bibitem{tomandvicky}
V.~Gould, Y.~Dandan, T.~Quinn-Gregson, and R.-E. Zenab.
\newblock Semigroups with finitely generated left congruence relation.
\newblock {\em submitted}, 2018.

\bibitem{green}
D.~G. Green.
\newblock The lattice of congruences on an inverse semigroup.
\newblock {\em Pacific J. Math.}, 57(1):141--152, 1975.

\bibitem{Scheiblich}
S.~H.
\newblock Kernels of inverse semigroup homomorphisms.
\newblock {\em J. Austral. Math. Soc.}, 18:289--292, 1974.

\bibitem{Howie}
J.~M. Howie.
\newblock {\em Fundamentals of semigroup theory}.
\newblock Oxford: Clarendon, 1995.

\bibitem{jones}
P.~R. Jones.
\newblock Distributive inverse semigroups.
\newblock {\em J. London Math. Soc.}, 2(3):457--466, 1978.

\bibitem{kozhukhov1980semigroups}
I.~B. Kozhukhov.
\newblock On semigroups with minimal or maximal condition on left congruences.
\newblock {\em Semigroup Forum}, 21(1):337--350, 1980.

\bibitem{Meakin}
J.~Meakin.
\newblock One-sided congruences on inverse semigroups.
\newblock {\em Trans. Amer. Math. Soc.}, 206:67--82, 1975.

\bibitem{nico}
W.~R. Nico.
\newblock A classification of indecomposable s-sets.
\newblock {\em J. Algebra}, 54(1):260--272, 1978.

\bibitem{petrich2}
M.~Petrich.
\newblock Congruences on inverse semigroups.
\newblock {\em J. Algebra}, 55:231--256, 1978.

\bibitem{PetrichRankin}
M.~Petrich and S.~Rankin.
\newblock The kernel-trace approach to right congruences on an inverse
  semigroup.
\newblock {\em Trans. Amer. Math. Soc.}, 330:917--932, 1992.

\bibitem{PRBrandt}
M.~Petrich and S.~Rankin.
\newblock Right congruences on a brandt semigroup.
\newblock {\em Soochow J. Math.}, 22(1):85--106, 1996.

\end{thebibliography}
\end{small}
\end{document}